\documentclass[a4paper]{article}

\usepackage{amsmath}

\usepackage{amsthm}
\usepackage[dvips]{graphicx}

\usepackage{amscd}
\usepackage{epsfig}
 
\usepackage{color}

\usepackage{psfrag}

\usepackage{amssymb}
\usepackage{latexsym}

 \usepackage{hyperref}
\hypersetup{%
  colorlinks=true,%
  linkcolor=blue,%
  citecolor=blue%
}

\newtheorem{Theorem}{Theorem}

\newtheorem{Proposition}[Theorem]{Proposition}
\newtheorem{Lemma}[Theorem]{Lemma}
\newtheorem{Corollary}[Theorem]{Corollary}
\newtheorem{Claim}[Theorem]{Claim}

\theoremstyle{definition} 
       \newtheorem{Definition}[Theorem]{Definition}

\theoremstyle{remark} 
 \newtheorem{Remark}[Theorem]{Remark}

\def\Fp{{\overline{\mathbb F}_p}}
\def\K{\mathbb K}

\begin{document}

\title{Non-standard components of the character variety for a family of
Montesinos knots}

\author{Luisa Paoluzzi and Joan Porti}
\date{\today}

\maketitle

\begin{abstract}
\vskip 2mm
We show that there are Montesinos knots with $n+1\geq 4$ tangles whose 
character varieties contain arbitrarily many irreducible components of 
dimension $d$ for any $1\le d\le n-2$. Moreover, these irreducible components 
can be chosen so that the trace of the meridian is non-constant.

\noindent\emph{AMS classification: } Primary 57M25; Secondary 20C99; 57M50.

\vskip 2mm

\noindent\emph{Keywords:} Character varieties, Montesinos knots, bending,
parabolic representations.

\end{abstract}

\section{Introduction}
\label{section:introduction}

The study of character varieties associated to representations of $3$-manifold 
groups into $SL_2(\mathbb{C})$ has received great attention in recent years.
Indeed, the understanding of the character variety of a manifold may give some 
insight on the structure of the manifold itself and notably on the existence of
essential surfaces, by means of Culler-Shalen theory \cite{CullerShalen}. On 
the other hand, little is known about $SL_2$-character varieties over fields of 
positive characteristic. It follows from work of Gonzalez-Acu\~na and 
Montesinos \cite{AcunaMontesinos}, that the defining polynomial equations for 
an $SL_2$-character variety -which have coefficients in ${\mathbb Z}$- are the 
same over any field of characteristic different from $2$. 

Standard results in algebraic geometry ensure that affine variety defined by
polynomials with coefficients in ${\mathbb Z}$ has the same geometric 
properties, like the dimension and the number of irreducible components, when 
considered over $\mathbb{C}$ or over an algebraically closed field of 
characteristic $p$, for almost all primes $p$. This follows basically from the
fact that the dimension of an affine variety (that is, the maximal dimension of
its irreducible components) and its irreducible components can be computed
algorithmically (see, for instance, \cite[Chapter 9]{CLoS} for the dimension, 
and \cite[page 209]{CLoS} for the decomposition into irreducible components).

In our situation, this says that for almost every prime $p>2$ the 
$SL_2$-character variety for a manifold over the algebraic closure of 
the prime field ${\mathbb F}_p={\mathbb Z}/p{\mathbb Z}$ looks precisely like 
the $SL_2$-character variety over $\mathbb{C}$. It is, however, possible that 
for some exceptional prime $p\neq 2$ this is not the case, and the number of 
its irreducible components or their dimensions might change. 

Our first motivation for wanting to establish the occurrence of this phenomenon 
comes from the fact that one may hope to find, in these special character 
varieties, ``new" curves whose ideal points are associated to essential 
surfaces which cannot be detected in characteristic $0$ (see  
\cite{SchanuelZhang} for examples of essential surfaces presenting this 
behaviour).

Although coming upon an actual example of surfaces detected only by curves in
characteristic $p$ appears to be extremely hard, work by Riley 
\cite{RileyI,RileyII} seems to provide evidence that this kind of phenomenon 
does happen, that is, there might be curves that appear only in certain 
characteristics. In his paper \cite{RileyI}, Riley studied parabolic 
representations (i.e. where all meridians are sent to matrices with trace $\pm 
2$) in characteristic $p$ for the group of a specific Montesinos knot with four 
tangles, and showed that the group admitted a one-parameter family of 
non-conjugate parabolic representations for each prime $p$.

The straightforward observation that in characteristic $p$ parabolic elements 
have order $p$ implies that the orbifold whose underlying topological space is 
the $3$-sphere and whose singular set is the given Montesinos knot with order 
of ramification equal to $p$ admits ``several" representations in 
characteristic $p$, and possibly ``more" than in characteristic $0$. We shall
give a precise meaning to this statement in Section~\ref{section:Fp}.

The above observation motivated our study of the character variety of this and  
other Montesinos knots in characteristic $0$. We are mainly interested in 
understanding the geometric reason behind the existence of Riley's 
representations, for its comprehension could lead to prove the existence of 
extra representations in other cases. This turns out to be related to the 
particular structure of the Montesinos knot considered by Riley, which allows 
to perform what Riley calls (in a subsequent paper \cite{RileyII}, again on 
parabolic representations of knots) the \emph{commuting trick}. The commuting 
trick boils down to the elementary remark that crossings between two arcs whose 
associated generators commute in a given representation are nugatory and can 
be arbitrarily changed. Of course, Riley was not the first to exploit this 
basic fact as Riley himself remarks, cf.~\cite{MagnusPeluso}.

It is well-known that an analysis of the character variety of the knot
can be carried out explicitly only for knots whose groups have a very limited 
number of generators, due to the computational complexity involved. It is thus 
helpful to find indirect methods to deduce properties (like the number of 
irreducible components, their dimensions, and their intersections) of the 
character variety. We will consider a family of Montesinos knots with at least 
$4$ tangles that we shall call \emph{Montesinos knots of Kinoshita-Terasaka 
type} (see Section~\ref{section:Montesinos} for a precise definition).

Using this elementary remark and \emph{bending} (see 
Section~\ref{section:bending}), we are able to prove for this class of knots 
the existence of irreducible components of large dimension in their character 
variety which are \emph{non-standard} in that they are different from the three 
\emph{standard} ones: the distinguished curve containing the holonomy 
character, the abelian component, and the Teichm\"uller components whose points 
are associated to representations of the base of the Seifert fibration of the 
orbifold whose underlying topological space is the $3$-sphere and whose 
singular set is the Montesinos knot with order of singularity equal to $2$. 
These Teichm\"uller components have dimension at most $n-2$ and, since the 
meridian is mapped to a hyperbolic isometry of order two, the trace of the 
meridian is constant equal to $0$ on them. Our first main result can thus be 
stated as follows:

\begin{Theorem}
\label{theorem:simplified}
Let $K$ be a Montesinos knot of Kinoshita-Terasaka type with $n+1$ tangles,
$n\ge3$. Its character variety contains (at least) two irreducible components 
of dimension $\ge n-2$ which are not contained in the hyperplane
defined by the condition that the trace of the meridian is equal to $0$.
\end{Theorem}

This resumes two stronger and more precise statements
(Theorems~\ref{theorem:parabolic} and \ref{theorem:non parabolic}) whose proofs 
will be provided in Sections~\ref{section:parabolic} and \ref{section:bound}. 
In particular, one can establish the precise dimension of these components and 
say something more on their number (see Theorems~\ref{theorem:non parabolic} 
and \ref{theorem:dim par}).

\begin{Theorem}
\label{theorem:many components}
For all integers $m>0$ and $n \ge 3$ there is a Montesinos knot of 
Kinoshita-Terasaka type with $n+1$ tangles whose character variety contains at 
least $m$ irreducible components of dimension $n-2$. 
\end{Theorem}

This improvement on Theorem~\ref{theorem:simplified} is a consequence of a
result by Ohtsuki, Riley and Sakuma on the character varieties of $2$-bridge
knots which shows that the number of their components can be arbitrarily large 
\cite{ORK}.

The same methods allow to prove a generalisation of Theorem~\ref{theorem:many 
components} in which the irreducible components of dimension $n-2$ are replaced 
by irreducible components of dimension $d$ for any $1\le d\le n-2$ (see again 
Theorem~\ref{theorem:non parabolic}). 

Although the non-standard components which are the object of
Theorem~\ref{theorem:simplified} are obtained by bending, as already observed,
the commuting trick is responsible for the existence of other non-standard
components: this is for instance the case of the \emph{r-components} detected 
by Mattman in the character variety of certain pretzel knots 
\cite[Thm 1.6]{Mattman}. The geometric interpretation behind the existence of 
Mattman's non-standard components will be briefly discussed in 
Section~\ref{section:Mattman}.

The paper is organised as follows. In Section~\ref{section:characters} we shall 
recall some basic facts about character varieties. The class of Montesinos 
knots we shall be dealing with will be introduced in 
Section~\ref{section:Montesinos}: there we shall also see that the knots in 
this class are closely related to connected sums of $2$-bridge knots. The main
feature of connected sums of knots is that they admit several representations
obtained from the representations of the single components by \emph{bending}: 
this procedure will be described in Section~\ref{section:bending}. 
Section~\ref{section:parabolic} and Section~\ref{section:non parabolic}
will be devoted to the construction respectively of non-standard components of 
parabolic characters and non-standard components of non-parabolic characters, 
whose number can be arbitrarily large thus proving 
Theorems~\ref{theorem:simplified} and \ref{theorem:many components}. The
contents of Section~\ref{section:bound} are more technical, and allow to
establish the exact dimension of these non-standard components: the analysis of
the parabolic components and the non-parabolic ones occupy a subsection each
(Subsections~\ref{subsect:nonparabolic} and 
\ref{section:more_on_intersections}). Finally, we shall discuss Mattman's 
non-standard components (Section~\ref{section:Mattman}) and comment on the 
character varieties of Montesinos knots of Kinoshita-Terasaka type over fields 
of positive characteristic (Section~\ref{section:Fp}).

\section{Character varieties}
\label{section:characters}

The variety of representations of a finitely generated group $G$ is the set of 
representations of $G$ in $SL_2(\mathbb{C})$:
$$
R(G)=\hom(G,SL_2(\mathbb{C})).
$$
Since $G$ is finitely generated, $R(G)$ can be embedded in a product 
$SL_2(\mathbb{C})\times\cdots\times SL_2(\mathbb{C})$ by mapping each 
representation to the image of a generating set. In this way $G$ is an affine 
algebraic set, whose defining polynomials are induced by the relations of a 
presentation of $G$ and whose coefficients are thus in $\mathbb{Z}$. By
considering Tietze transformations, it is not hard to see that this structure 
is independent of the choice of presentation of $G$ up to isomorphism, 
cf.~\cite{LubotzkyMagid}. Note that what stated above remains valid if 
$\mathbb{C}$ is replaced by any other field $\K$, that we shall assume to be 
algebraically closed for simplicity. In particular, the defining relations for 
$R(G)$ are the same over every field. We shall write $R(G)_\K$ whenever we wish 
to stress that we are considering representations in $SL_2(\K)$. When the
subscript $ _\K$ is omitted, by convention $\K=\mathbb{C}$.

Given a representation $\rho\in R(G)$, its character is the map
$\chi_{\rho}:G\to \mathbb C$ defined by 
$\chi_{\rho}(\gamma)=\operatorname{trace}(\rho(\gamma))$, $\forall\gamma\in G$. 
The set of all characters is denoted by $X(G)$.

Given an element $\gamma\in G$, we define the map
$$
\begin{array}{rcl}
     \tau_\gamma: X(G) & \to & \mathbb C \\
              \chi & \mapsto & \chi(\gamma)
\end{array}.
$$

\begin{Proposition}[\cite{CullerShalen,AcunaMontesinos}]
\label{proposition:X(G)}
The set of characters $X(G)$ is an affine algebraic set defined over 
$\mathbb Z$, which embeds in $\mathbb C^N$ with coordinate functions
$(\tau_{\gamma_1}, \ldots, \tau_{\gamma_N})$ for some 
$\gamma_1,\ldots,\gamma_N\in G$.  
\end{Proposition}

The affine algebraic set $X(G)$ is called the \emph{character variety} of $G$:
it can be interpreted as the algebraic quotient of $R(G)$ by the conjugacy 
action of $PSL_2(\mathbb{C})=SL_2(\mathbb{C})/\mathcal{Z}(SL_2(\mathbb{C}))$.

Note that the set $\{\gamma_1,\ldots,\gamma_N\}$ in the above proposition can 
be chosen to contain a generating set of $G$. For $G$ the fundamental group of 
a knot exterior, we will then assume that it always contains a representative 
of the meridian.

A careful analysis of the arguments in \cite{AcunaMontesinos} shows that 
Proposition~\ref{proposition:X(G)} still holds if $\mathbb C$ is replaced by
any algebraically closed field, provided that its characteristic is different
from $2$. Let $\mathbb F_p$ denote the field with $p$ elements and $\Fp$ its
algebraic completion. We have:

\begin{Proposition}[\cite{AcunaMontesinos}]
\label{proposition:X(G)Fp}
Let $p>2$ be an odd prime number. The set of characters $X(G)_{\Fp}$ associated 
to representations of $G$ over the field $\Fp$ is an algebraic set which embeds 
in $\Fp ^N$ with the same coordinate functions $(\tau_{\gamma_1}, \ldots, 
\tau_{\gamma_N})$ seen in Proposition~\ref{proposition:X(G)}. Moreover,
$X(G)_\Fp$ is defined by the same polynomials over $\mathbb Z$ as 
$X(G)_\mathbb{C}$.  
\end{Proposition}

A representation $\rho\in R(G)$ is called \emph{irreducible} if no proper 
subspace of $\mathbb C^2$ is $\rho(G)$-invariant. The set of irreducible 
representations is Zariski open, and so is the set of irreducible characters
\cite{CullerShalen}. We denote them by $R_{\rm irr}(G)$ and $X_{\rm irr}(G)$ 
respectively. The following lemma is proved by Culler and Shalen and 
Gonz\'alez-Acu\~na and Montesinos in \cite{CullerShalen,AcunaMontesinos} for 
$\mathbb C$.

\begin{Lemma}[\cite{CullerShalen,AcunaMontesinos}]
\label{lemma:varietycharacters} 

The projection 
$$
\begin{array}{rcl}
 R(G) & \to & X(G) \\
 \rho & \mapsto & \chi_{\rho}
\end{array}
$$
is surjective.
Moreover  $R_{\rm irr}(G)\to X_{\rm irr}(G)$ is a local fibration with fibre 
the orbit by conjugacy.

\end{Lemma}

The following is well-known for $\mathbb C$, but the same proof applies to 
$\Fp$.

\begin{Lemma}
\label{lemma:Freerank2} Let $\mathbb K=\mathbb C$ or $\Fp$ for $p\neq 2$.
For a free group on two generators $F_2=\langle \gamma_1,\gamma_2\mid \rangle$,
$X(F_2)_{\mathbb K}\cong \mathbb K^3$ with coordinates 
$(\tau_{\gamma_1},\tau_{\gamma_2},\tau_{\gamma_1\gamma_2})$.
\end{Lemma}

For a compact manifold $M$, we use the notation $R(M)=R(\pi_1(M))$ and 
$X(M)=X(\pi_1(M))$. For a knot $K\subset S^3 $, we write, 
$R(K)=R(S^3\setminus \mathcal N(K))$ and $X(K)=X( S^3\setminus \mathcal N(K))$,
where $\mathcal N(K)$ denotes an open regular neighbourhood of $K$.

Recall that by \cite[Corollary~3.3]{BoileauZimmermannpi} the fundamental group 
of a knot is generated by two meridians if and only if it is a $2$-bridge knot 
(see Section~\ref{section:Montesinos}).

\begin{Corollary}
\label{cor:plane curve}
Assume that $K\subset S^3$ is a $2$-bridge knot, that is its fundamental group 
is generated by two meridians:
$\pi_1(S^3\setminus K)=\langle \mu_1,\mu_2\mid r_i\rangle$. Then, for 
$\mathbb K=\mathbb C$ or $\Fp$, $X( K)_{ \mathbb K}$ is a plane curve with 
coordinates $\tau_{\mu_1}$ and $\tau_{\mu_1\mu_2}$.
\end{Corollary}

This uses Lemma~\ref{lemma:Freerank2} and the fact that  
$\tau_{\mu_1}=\tau_{\mu_2}$, because $\mu_1$ and $\mu_2$ are conjugate. 
Moreover, by a theorem of Thurston \cite{ThurstonNotes} (see also 
\cite{KapovichBook}), each irreducible component has to be at least a curve.

Sometimes it will be convenient to work with $PSL_2(\mathbb{C})$ instead of 
$SL_2(\mathbb{C})$. In this case we use the notation $R(M,PSL_2(\mathbb{C}))$
for the representation variety while its quotient in invariant theory by 
conjugacy will be denoted by $X(M,PSL_2(\mathbb{C}))$ 
(cf.~\cite{BoyerZhang,HeusenerPorti} for an interpretation in terms of 
characters).

\begin{Proposition}[\cite{Goldman}]
\label{proposition:dimsurface}
Let $\mathcal O^2$ be a compact two dimensional orbifold with $b$ cone points 
and $c$ corners. If $e$ denotes the Euler characteristic of the underlying 
surface $\vert \mathcal O^2\vert$, then 
$$
\dim X(\mathcal O^2, PSL_2(\mathbb{C}))= -3 e+2 b+c.
$$
\end{Proposition}

Here, $\dim X(\mathcal O^2, PSL_2(\mathbb{C}))$ means the maximal dimension of 
the irreducible components of $X(\mathcal O^2, PSL_2(\mathbb{C}))$.

\section{Montesinos knots of Kinoshita-Terasaka type}
\label{section:Montesinos}

The exposition in this section follows roughly the presentation in Zieschang's 
paper \cite{Zieschang}.

Recall that a \emph{rational tangle} is any two-string tangle that can be
obtained from the trivial tangle (i.e. two unknotting vertical arcs running 
parallel from the bottom to the top of a ball seen as a cube) by an isotopy of 
the ball which does not leave its boundary pointwise fixed. The general form of 
a rational tangle is shown in Figure~\ref{fig:rattangle} where the labels 
$a'_i$, $a''_i$ and $a_k$ denote the number of positive crossings, with the 
convention that a negative crossing counts for $-1$ positive crossings. It can 
be shown that the continued fraction 
$\frac{\beta}{\alpha}=\frac{1}{a_1+\frac{1}{-a_2+\dots}}$, where 
$a_i=a'_i+a''_i$ for $i=1,\dots,k-1$ is an invariant of the isotopy class 
of the rational tangle, where isotopies in this case are required to leave the 
boundary pointwise fixed.

\begin{figure}[h]
\begin{center}
 {
 \psfrag{a1'}{$a_1'$}
 \psfrag{b1}{${a_1''}$}
 \psfrag{a2'}{$a_2'$}
 \psfrag{b2}{${a_2''}$}
 \psfrag{a3'}{$a_3'$}
 \psfrag{b3}{${a_3''}$}
 \psfrag{ak}{$a_k$}
  \includegraphics[height=5cm]{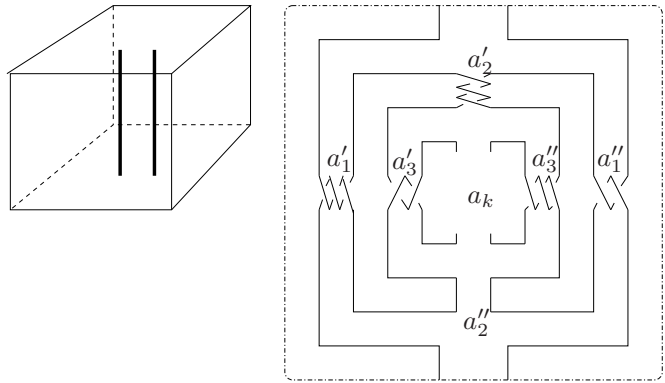}
 }
\end{center}
\caption{A trivial tangle, and a rational tangle in its standard form.}
\label{fig:rattangle}
\end{figure}

Rational tangles are closely related to \emph{$2$-bridge knots} and
\emph{links}. These are links obtained by gluing together two trivial tangles
along their boundaries, or equivalently by closing up a rational tangle by 
adding two arcs, one connecting the bottom ends of the tangle and one 
connecting its top ends (See Figure~\ref{fig:bridge}). We shall denote by 
$B(\frac{\beta}{\alpha})$ the $2$-bridge link obtained by closing the rational 
tangle with invariant $\frac{\beta}{\alpha}$.

Montesinos links can be interpreted as a generalization of $2$-bridge links in
which several tangles are stacked together one after the other in a circular
pattern as shown in Figure~\ref{fig:montesinos}, the $2$-bridge link case 
corresponding to the situation where a unique tangle is used (see 
Figure~\ref{fig:bridge}). Note, though, that the tangle must be rotated of 
$\pi/2$ for the two constructions to be consistent; in particular the two 
continued fractions for the $2$-bridge and the Montesinos presentations give 
rational numbers which are opposite of inverses of one another. It was proved 
by Bonahon (cf.~\cite{BonahonSiebenmann}) that the a Montesinos link with 
$n\ge 3$ tangles is completely determined by the ordered set of the $n$ 
rational numbers $\frac{\beta_i}{\alpha_i}\in(0,1)$ associated to its $n$ 
tangles up to cyclic permutations and reverse of order, together with the 
number $e_0=e-\sum_{i=1}^n\frac{\beta_i}{\alpha_i}$, where $e$ is the number of
crossings that appear outside the $n$ tangles (see 
Figure~\ref{fig:montesinos}). Note that these extra crossings can be englobed 
in the rational tangles if we do not require their associated continued 
fractions to belong to $(0,1)$.

\begin{figure}[h]
\begin{center}
 {
 \psfrag{b}{${\beta}/{\alpha}$}
  \includegraphics[height=3cm]{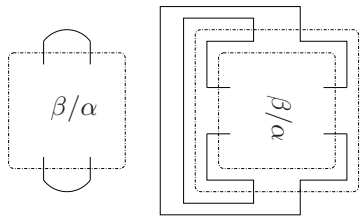}
 }
\end{center}
\caption{A $2$-bridge knot and its Montesinos form; here $\frac{\beta}{\alpha}$
denotes the rational value of the continued fraction associated to the rational 
tangle.}
\label{fig:bridge}
\end{figure}

\begin{figure}[h]
\begin{center}
 {
 \psfrag{b1}{$\beta_1/\alpha_1$}
 \psfrag{b2}{$\beta_2/\alpha_2$}
 \psfrag{bn}{$\beta_n/\alpha_n$}
 \psfrag{e}{$e$}
  \includegraphics[height=5cm]{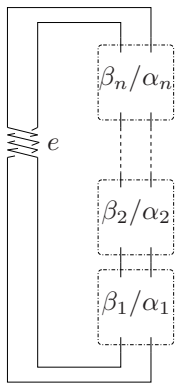}
 }
\end{center}
\caption{A Montesinos link with $n$ rational tangles.}
\label{fig:montesinos}
\end{figure}

It is not hard to see that a Montesinos link is a knot if and only if either
there is a unique even $\alpha_i$ or there is no even $\alpha_i$ and in this
case the $\beta_i$s and $e$ must satisfy some extra condition (see 
Boileau-Zimmermann \cite[Fig. 4, page 570]{BoileauZimmermann}). Note that when 
$\alpha_i$ is even, each arc of the $i$-th tangle enters and exits the ball on 
the same side.

\begin{Definition}
\label{definition:KTtype}
A Montesinos knot with $n+1$ tangles will be called \emph{of Kinoshita-Terasaka
type}\footnote[1]{The expression is already used by Riley in his paper: these 
knots can in fact be seen as a Kinoshita-Terasaka sum of a $2$-component link, 
see \cite{RileyI}} if $\alpha_{n+1}$ is even.
\end{Definition} 

According to the previous discussion, the $\alpha_i$s are all odd for $1\le 
i\le n$. We can furthermore assume that $e=0$, up to allowing
$\frac{\beta_{n+1}}{\alpha_{n+1}}$ to be an arbitrary rational. For $n\ge 3$ we
shall denote by $M(\frac{\beta_1}{\alpha_1},\dots,\frac{\beta_n}{\alpha_n},
\frac{\beta_{n+1}}{\alpha_{n+1}})$ the Montesinos knot of Kinoshita-Terasaka
type obtained by stacking together $n+1$ rational tangles of invariants
$\frac{\beta_i}{\alpha_i}$, satisfying the aforementioned requirements.

Now let $K=M(\frac{\beta_1}{\alpha_1},\dots,\frac{\beta_n}{\alpha_n},
\frac{\beta_{n+1}}{\alpha_{n+1}})$ and let $K'$ be the composite knot whose
prime summands are the $n$ $2$-bridge knots $B(\frac{\beta_i}{\alpha_i})$,
$i=1,\dots,n$, i.e. $K'=B(\frac{\beta_1}{\alpha_1})\sharp\dots\sharp
B(\frac{\beta_n}{\alpha_n})$. Note that $K'$ can be obtained from $K$ by
changing some crossings in the $n+1$st tangle. 

For each of the two knots, consider the meridians $\mu_i$ and $\mu'_i$, $1\le i
\le n+1$, as shown in Figure~\ref{fig:composite}: they generate the fundamental 
groups of (the exteriors of) $K$ and $K'$, and are in fact a redundant system 
of generators. This follows from the fact that the fundamental group of the 
(exterior) of a rational tangle is a free group of rank $2$; in particular, the 
fundamental group of the $i$th tangle is generated by two meridians among 
$\mu_i$, $\mu'_i$, $\mu_{i+1}$ and $\mu'_{i+1}$. Using Wirtinger's method, one 
can deduce the following presentations for the fundamental groups of $K$ and 
$K'$:
$$\pi_1(K)=\langle \mu_1,\mu_1',\dots,\mu_{n+1},\mu_{n+1} \mid {\mathcal R},
w_1\mu_1=\mu'_1w_1, w_{n+1}\mu_{n+1}=\mu'_{n+1}w_{n+1}\rangle,$$
and
$$\pi_1(K')=\langle \mu_1,\mu_1',\dots,\mu_{n+1},\mu_{n+1} \mid {\mathcal R},
\mu_1=\mu'_1, \mu_{n+1}=\mu'_{n+1}\rangle,$$
where ${\mathcal R}$ is a set of $2n$ relations expressing, for each
$i=1,\dots,n$, two meridians among $\mu_i$, $\mu'_i$, $\mu_{i+1}$ and
$\mu'_{i+1}$ as conjugates of the other two, and are obtained from the 
Wirtinger relations inside the $i$th tangle. Similarly $w_1$ and $w_{n+1}$ are 
products of the elements $\mu_1$, $\mu'_1$, $\mu_{n+1}$ and $\mu'_{n+1}$, and
the last two relations are obtained from the Wirtinger relations in the $n+1$st 
tangle. Note that, for each $i$, the meridians $\mu_i$ and $\mu'_i$ cobound an 
annulus in the exterior of $K'$, so $\mu_i=\mu'_i$ in $\pi_1(K')$ for all $i$. 
Note, moreover, that the fundamental group of the composite knot $K'$ can also 
be described as sum of the fundamental groups of its summands amalgamated over 
cyclic subgroups:
$$
\pi_1(K')\cong\pi_1(B(\beta_1/\alpha_1))*_{\mathbb Z}\dots*_{\mathbb
Z}\pi_1(B(\beta_n/\alpha_n)),
$$ 
where the amalgamating subgroups ${\mathbb Z}$ are generated by the meridians
$\mu_i=\mu'_i$, $i=2,\dots,n$.
 
\begin{figure}[h]
\begin{center}
 {
 \psfrag{b1}{$\beta_1/\alpha_1$}
 \psfrag{b2}{$\beta_2/\alpha_2$}
 \psfrag{bn}{$\beta_n/\alpha_n$}
 \psfrag{bn1}{$\beta_{n+1}/\alpha_{n+1}$}
 \psfrag{n1}{$\mu_{1}'$}
 \psfrag{n2}{$\mu_{2}'$}
 \psfrag{n3}{$\mu_{3}'$}
 \psfrag{nn}{$\mu_{n}'$}
 \psfrag{nm}{$\mu_{n+1}'$}
 \psfrag{m1}{$\mu_{1}$}
 \psfrag{m2}{$\mu_{2}$}
 \psfrag{m3}{$\mu_{3}$}
 \psfrag{mn}{$\mu_{n}$}
 \psfrag{mm}{$\mu_{n+1}$}
 \psfrag{K}{$K$}
 \psfrag{Kk}{$K'$}
  \includegraphics[height=10cm]{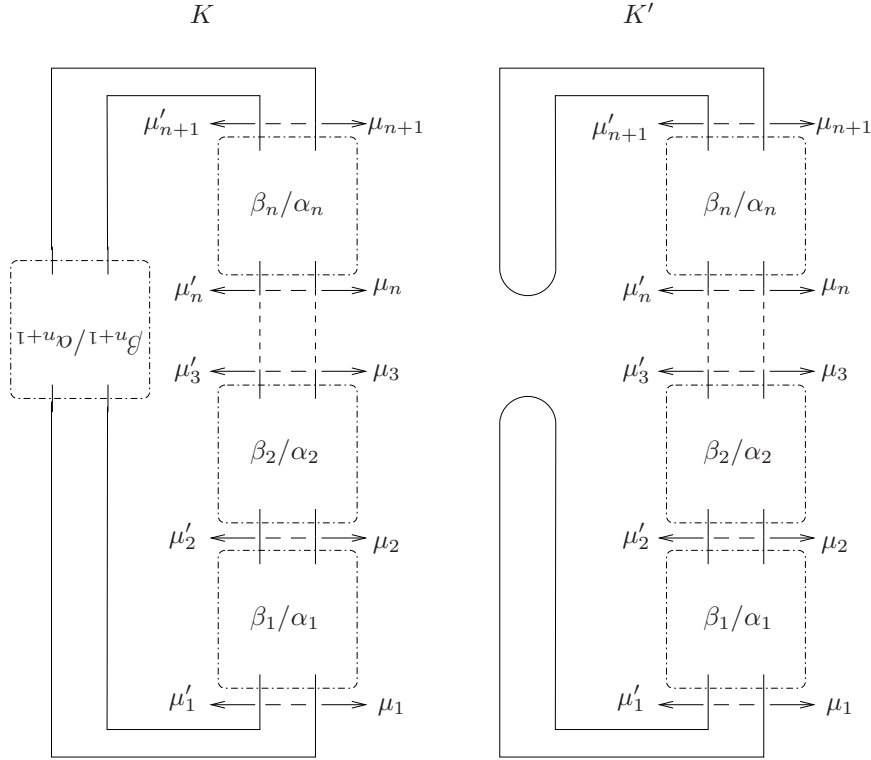}
 }
\end{center}
\caption{A Montesinos knot of Kinoshita-Terasaka type on the left and its
associated composite knot on the right.}
\label{fig:composite}
\end{figure}

The following result is straightforward, in view of the above presentations:

\begin{Proposition}
\label{proposition:quotient}
The groups obtained by quotienting $\pi_1(K)$ and $\pi_1(K')$ by the normal 
subgroups $H$ and $H'$ respectively, both normally generated by the six 
commutators $[\mu_1,\mu'_1]$, $[\mu_1,\mu_{n+1}]$, $[\mu_1,\mu'_{n+1}]$, 
$[\mu'_1,\mu_{n+1}]$, $[\mu'_1,\mu'_{n+1}]$, and $[\mu_{n+1},\mu'_{n+1}]$, are 
isomorphic. In particular, the representation variety of the group 
$\Gamma=\pi_1(K)/H \cong \pi_1(K')/H'$ is a subvariety of both representation 
varieties for $\pi_1(K)$ and $\pi_1(K')$. The analogue conclusion holds for the 
character varieties.
\end{Proposition}

\begin{Remark}
\label{remark:small components}
Let $\hat K$ be a Montesinos knot obtained from $K$ by deleting some of its
rational tangles $\frac{\beta_i}{\alpha_i}$ corresponding to indices $i\le n$,
and let $\hat K'$ be its associated composite knot. It is easy to see that 
every representation of $\hat K'$ extends to a representation of $K'$ whose
restriction to the ``missing" summands is the obvious abelian one. As a
consequence, if $\hat\Gamma$ denotes the common quotient of $\pi_1(\hat K)$ and
$\pi_1(\hat K')$, as defined in this section, we have that 
$X(\hat\Gamma)\subset X(\Gamma)$.
\end{Remark}

\section{Bending}
\label{section:bending}

Recall that $\pi_1(K')$ is the amalgamated product of 
$\pi_1(B(\frac{\beta_1}{\alpha_1}))$, $\ldots$, 
$\pi_1(B(\frac{\beta_n}{\alpha_n}))$ along the cyclic groups generated by 
$\mu_2,\ldots \mu_n$.

Let $\rho\in R( K')$ be a non-trivial representation. Denote by 
$A_i< PSL_2(\mathbb C)$ the projection of the centraliser of $\rho(\mu_i)$ (and
$\rho(\mu_i')$) in $SL_2(\mathbb C)$, for $i=2,\ldots,n$. By hypothesis
$\rho(\mu_i)$ is non-trivial. For the centraliser $A_i$ we have:
\begin{itemize}
\item If $\rho(\mu_i)$ is parabolic, then $A_i\cong \mathbb C$ is a parabolic 
group, that stabilises the same point of $\partial_{\infty}\mathbf H^3 \cong 
\mathbb P^1(\mathbb C)\cong\hat{\mathbb C}$ as $\rho(\mu_i)$.
\item If $\rho(\mu_i)$ is hyperbolic or elliptic (i.e. 
$trace(\rho(\mu_i))\neq \pm 2$), then $A_i\cong \mathbb C^*$ is the group that 
preserves the same oriented geodesic as $\rho(\mu_i)$.
\end{itemize}

If $a=(a_2,\dots,a_{n})\in A_2\times\dots\times A_{n}$ where 
$A_i\subset PSL_2({\mathbb C})$ is the projection of the centraliser in 
$SL_2({\mathbb C})$ of $\rho_i(\mu_{i+1})=\rho_{i+1}(\mu_{i+1})$, then $a\rho$ 
is the representation defined as 
$$a\rho_=\rho_1*\, ^{a_2}\rho_2*\, ^{a_2a_3}\rho_3*\dots *\,
^{a_2a_3\dots a_{n}}\rho_n,$$ 
where $^x\rho$ is the representation obtained by conjugating $\rho$ by 
$x\in PSL_2(\mathbb C)$.

By \cite[Lemma 5.6]{JohnsonMillson} we have:

\begin{Lemma}[Johnson and Millson \cite{JohnsonMillson}]
\label{lemma:bending}
If the restriction of $\rho \in R( K')$ to each 
$\pi_1(B(\frac{\beta_i}{\alpha_i}))$ is irreducible, then for 
$a, b\in A_2\times\dots\times A_{n}$ in a neighborhood of the identity 
$a\rho$ is conjugate to $b\rho$ if and only if $a=b$.
\end{Lemma}

Consider the  map whose components are the restriction of characters to each 
$\pi_1(B(\frac{\beta_i}{\alpha_i}))$:
$$
\pi:X(K')\to  X(B(\tfrac{\beta_1}{\alpha_1}))\times\cdots\times 
X(B(\tfrac{\beta_n}{\alpha_n})).
$$

\begin{Corollary}
\label{corollary:bending}
Let $(\chi_1,\ldots,\chi_n)\in X(B(\tfrac{\beta_1}{\alpha_1}))
\times\cdots\times X(B(\tfrac{\beta_n}{\alpha_n})) $. If non-empty, the fibre 
$\pi^{-1}(\chi_1,\ldots,\chi_n)$ has dimension $\geq s-1$, where $s$ is the 
number of irreducible characters among $\chi_1,\ldots,\chi_n$.
\end{Corollary}

\begin{proof}
Assume first that all the $\chi_i$ are irreducible, i.e. $s=n$. Then the 
dimension of the fibre is $\geq n-1$ by Lemmas~\ref{lemma:bending} and  
\ref{lemma:varietycharacters}. For an arbitrary $s\geq 1$, we just use the 
argument of Remark~\ref{remark:small components}.
\end{proof}

Note that if all $\chi_i$ are irreducible, then the inequality in 
Corollary~\ref{corollary:bending} is an equality by 
Lemma~\ref{lemma:varietycharacters}. We now want to give also an upper bound on 
the dimension of the fibre in the general case. For that we need to understand 
the reducible characters. We begin with the case where all $\chi_i$ are 
parabolic, and start with a remark:

\begin{Remark}
\label{remark:abelian_par}
A parabolic representation of $\pi_1(B(\frac{\beta}{\alpha} ))$ that is 
reducible is also abelian and, up to conjugacy, it maps 
$\gamma\in \pi_1(B(\frac{\beta}{\alpha} ))$ to
$$
 \pm \begin{pmatrix}
     1 & h(\gamma) \\
      0 & 1
    \end{pmatrix}
$$
for some homomorphism $h: \pi_1(B(\frac{\beta}{\alpha} ))\to\mathbb C$. 
Its character is the trivial one, though the representation may be non-trivial.
\end{Remark}

Since $\pi_1(K')$ is normally generated by any meridian, we are only interested
in the case where the previous $h$ is nontrivial.

Thus when a character $\chi_i$ is parabolic and reducible, then $\chi_i$ is the 
character of an abelian representation $\rho_i$ for which 
$\rho_i(\mu_i)=\rho_i(\mu_{i+1})$. It follows that the global representation 
$\rho\in R(K')$ is in fact a representation of some $\hat K'$ in which the 
$i$-th tangle is omitted (see Remark~\ref{remark:small components}). 
Therefore we have:

\begin{Corollary}
\label{corollary:bending_par}
Let $(\chi_1,\ldots,\chi_n)\in 
X(B(\tfrac{\beta_1}{\alpha_1}))\times\cdots\times 
X(B(\tfrac{\beta_n}{\alpha_n}))$ be parabolic (i.e. $\chi_i$ of the meridian is 
$\pm 2$). If non-empty, the fibre $\pi^{-1}(\chi_1,\ldots,\chi_n)$ has 
dimension precisely $s-1$, where $s$ is the number of irreducible characters 
among $\chi_1,\ldots,\chi_n$.
\end{Corollary}

For non-parabolic characters that are reducible, we have to distinguish between 
those that are the character of only abelian representations and those that are 
also the character of (reducible) non-abelian ones.

The following lemma is due to \cite{Burde} and \cite{deRham}, cf \cite{HPS}.

\begin{Lemma} 
\label{lemma:rootsAlex}
Let $\chi\in X(B(\tfrac{\beta}{\alpha}))$ be a reducible character. Then the 
following are equivalent:
\begin{itemize}
\item[(i)] $\chi$ is the character of a non-abelian representation 
(besides abelian ones).
\item[(ii)] $\chi$ belongs to the Zariski closure of an irreducible component 
of $X(B(\tfrac{\beta}{\alpha}))$ that contains irreducible characters.
\item[(iii)] $\chi(\mu)=\theta+1/\theta$, where $\theta^2$ is a root of the 
Alexander polynomial of the knot and $\mu$ a meridian.
\end{itemize}
\end{Lemma}

\begin{Definition}
A character in $X(B(\tfrac{\beta}{\alpha}))$ is called \emph{generic reducible} 
if it is reducible and does not satisfy the assertions of 
Lemma~\ref{lemma:rootsAlex}.
\end{Definition}

By Lemma~\ref{lemma:rootsAlex}, since $\pm 1$ is not a root of the Alexander 
polynomial, when $\chi({\mu})=0$ or $\pm 2$, if $\chi$ is reducible then it is 
generic reducible (i.e. it is not the character of a reducible non-abelian 
representation). Hence as in Corollary~\ref{corollary:bending_par}, we have:

\begin{Corollary}
\label{corollary:bending_generic}
Let $(\chi_1,\ldots,\chi_n)\in 
X(B(\tfrac{\beta_1}{\alpha_1}))\times\cdots\times 
X(B(\tfrac{\beta_n}{\alpha_n})) $. Assume that if $\chi_i$ is reducible then it 
is generic reducible (i.e. it does not satisfy the assertions of 
Lemma~\ref{lemma:rootsAlex}). If non-empty, the fibre
$\pi^{-1}(\chi_1,\ldots,\chi_n)$ has dimension precisely $s-1$, where $s$ is 
the number of irreducible characters among $\chi_1,\ldots,\chi_n$.
\end{Corollary}

\section{Parabolic representations}
\label{section:parabolic}

In this section we shall define parabolic representations for the Montesinos
knot of Kinoshita-Terasaka type $M(\frac{\beta_1}{\alpha_1},\dots,
\frac{\beta_n}{\alpha_n},\frac{\beta_{n+1}}{\alpha_{n+1}})$ which are induced
by parabolic representations of the $n$ $2$-bridge knots
$B(\frac{\beta_i}{\alpha_i})$, $i=1,\dots,n$.

\begin{Definition}
\label{definition:parabolic variety}
We shall denote by $X_{\rm par}(K)$ the subvariety of the character variety of
a knot $K$ consisting of characters associated to \emph{parabolic 
representations}, i.e. those in which the meridian $\mu$ of $K$ is mapped to a 
parabolic matrix: $$X_{\rm par}(K)=\{\chi \in X(K)\mid \chi(\mu)=\pm 2\}.$$
Note that $X_{\rm par}(K)$ is never empty for it always contains at least the
trivial character (associated to the trivial representation).
\end{Definition}

From now on $K$ will denote the Montesinos knot
of Kinoshita-Terasaka type
 $M(\frac{\beta_1}{\alpha_1},
\dots, \frac{\beta_n}{\alpha_n},\frac{\beta_{n+1}}{\alpha_{n+1}})$ and $K'$ its
associated composite knot as defined in Section~\ref{section:Montesinos}.

For each $i=1,\dots,n$, let
$\rho_i:\pi_1(B(\frac{\beta_i}{\alpha_i}))\longrightarrow SL_2({\mathbb C})$ be
an irreducible parabolic representation of the $2$-bridge knot
$B(\frac{\beta_i}{\alpha_i})$. The existence of such a $\rho_i$ can be seen as
follows: $\pi_1(B(\frac{\beta_i}{\alpha_i}))$ admits an irreducible
representation in $PSL_2({\mathbb C})$ which corresponds to the holonomy
representation of $B(\frac{\beta_i}{\alpha_i})$ if the $2$-bridge knot is
hyperbolic or to the holonomy representation of the base of its fibration if
$B(\frac{\beta_i}{\alpha_i})$ is a torus knot. It then suffices to lift this
irreducible representation to $SL_2({\mathbb C})$ by choosing the same trace 
sign for each generator. This is consistent because all generators are
conjugate and because of the very nature of the Wirtinger's relations.

Since all non-diagonal parabolic matrices belong to two conjugacy classes
according to the sign of their trace, up to conjugacy, we can assume that 
$\rho_i(\mu_{i+1})=\rho_{i+1}(\mu_{i+1})$ for all $i=1,\dots,n-1$.
One can thus define a representation 
$$\rho=\rho_1*\dots*\rho_n:\pi_1(K')\longrightarrow SL_2({\mathbb C});$$
remark that in fact one can define several different representation in this
way, just by conjugating $\rho_{i+1}$ by an element in the centraliser of
$\rho_i(\mu_{i+1})=\rho_{i+1}(\mu_{i+1})$ as discussed in the previous section.

We now want to show that one can find representations of this kind which belong
to $X_{\rm par}(\Gamma)$, where $\Gamma$ is the common quotient of $\pi_1(K)$
and $\pi_1(K')$ defined in Section~\ref{section:Montesinos}. Recall from the
previous section that if $a=(a_1,\dots,a_{n-1})\in A_1\times\dots\times
A_{n-1}$ where $A_i$ is the projection in $PSL_2({\mathbb C})$ of the
centraliser of $\rho_i(\mu_{i+1})=\rho_{i+1}(\mu_{i+1})$, then $a\rho$ is the 
representation defined as $\rho_1*\, ^{a_1}\rho_2*\, ^{a_1a_2}\rho_3*\dots *\,
^{a_1a_2\dots a_{n-1}}\rho_n$, where $^x\rho$ is the representation obtained by
conjugating $\rho$ by $x$. We have

\begin{Lemma}
\label{lemma:rep par Gamma}
For each $\rho$, there is an $a\in A_1\times\dots\times A_{n-1}$, such that 
$a\rho(\mu_1)$ and $a\rho(\mu_{n+1})$ commute.
\end{Lemma} 

\begin{proof}
Recall that the subgroup $\pi_1(B(\frac{\beta_i}{\alpha_i}))$ of $\pi_1(K')$ is
generated by $\mu_i$ and $\mu_{i+1}$, for $i=1,\dots, n$. This implies that
that $\rho_i(\mu_i)$ and $\rho_i(\mu_{i+1})$ cannot belong to the same
reducible subgroup of $SL_2({\mathbb C})$ for $\rho_i$ is irreducible. If we
consider the natural action of $SL_2({\mathbb C})$ on ${\mathbb C}{\mathbb
P}^1\cong\hat {\mathbb C} = \mathbb C\cup\{\infty\}$, this is equivalent to say that the fixed points of 
$\rho_i(\mu_i)$ and $\rho_i(\mu_{i+1})$ are different. Note also that the 
centraliser $A_i$ of $\rho_i(\mu_i)$ acts transitively on 
$\hat{\mathbb C}\setminus Fix(\rho_i(\mu_i))$.

We shall start by proving the lemma for $n=3$. Without loss of generality, we
may assume that the fixed point of $\rho_2(\mu_2)$ is $0$ and that of
$\rho_2(\mu_3)$ is $\infty$. We can choose an element $a_1\in A_1$ and an
element $a_2\in A_2$ such that the fixed points of $^{a_1}\rho_1(\mu_1)$ and
that of $^{a_2}\rho_3(\mu_4)$ are both equal to -say- $1$. This is possible 
because the fixed point of $\rho_1(\mu_1)$ is in $\hat{\mathbb C}\setminus
\{0\}$ and that of $\rho_3(\mu_4)$ is in $\hat{\mathbb C}\setminus\{\infty\}$.
In fact, one can choose the common fixed point for $^{a_1}\rho_1(\mu_1)$ and
$^{a_2}\rho_3(\mu_4)$ to be any point in ${\mathbb C}\setminus \{0\}$. It is 
now evident that $\rho_1(\mu_1)$ and $\rho_3(\mu_4)$ commute, because they are 
parabolic elements fixing the same point. The desired representation is then 
$^{a_1}\rho_1*\,\rho_2*\,^{a_2}\rho_3$, which is conjugate to $a\rho$ with 
$a=(a_1^{-1},a_2)$.

Assume now that $n>3$. The same argument applies using $\mu_n$ instead of
$\mu_3$. Note that, for the argument to work, there is no need for the fixed
points of $\rho_2(\mu_2)$ and of $\rho_{n-1}(\mu_n)$ to be distinct. In fact,
in the case when they coincide, one only needs to conjugate $\rho_n$ (and not
$\rho_1$).
\end{proof}

We already knew that the intersection $X_{\rm par}(K')\cap X(\Gamma)$ is not
empty, for it contains the trivial character. The previous lemma shows moreover
that this intersection contains an irreducible character. The bending procedure
seen in Section~\ref{section:bending} assures that this irreducible character
is contained in an irreducible component $Y'$ of $X_{\rm par}(K')$ of dimension 
at least $n-1$. It is now easy to see that the subvariety $X_{\rm par}(K')\cap 
X(\Gamma)$ is obtained by intersecting $X_{\rm par}(K')$ with the hypersurface
defined by the equation $\chi([\rho(\mu_1),\rho(\mu_{n+1})])=2$. It is indeed
elementary to see that the commutator of two parabolic elements is trivial if
and only if its trace is equal to $2$. As a consequence the dimension of 
$Y'\cap X(\Gamma)$ is at least $n-2$ since $Y'\cap X(\Gamma)$ is non-empty.

The above considerations together with the fact, seen in
Section~\ref{section:Montesinos}, that $X(\Gamma)\subset X(K)$ give the 
following:

\begin{Theorem}
\label{theorem:parabolic}
Let $K$ be a Montesinos knot of Kinoshita-Terasaka type with $n+1$ tangles,
$n\ge3$. The subvariety $X_{\rm par}(K)$ contains a parabolic component of 
dimension at least $n-2$. 

In addition, the number of such components can be arbitrarily large.
\end{Theorem}

For the last assertion, we use that the parabolic component of a 
$2$-bridge knot consists of finitely many points, but its cardinality can be 
arbitrarily large by \cite{ORK}.

Remark that the irreducible components described in the above theorem 
correspond to the components studied by Riley for a specific Montesinos knot of 
Kinoshita-Terasaka type with $4$ tangles over fields of positive 
characteristic in \cite{RileyI}. More precisely, Riley only considered the 
$\mathbb F_p$-rational points corresponding to homomorphisms of the knot group 
to the finite group $SL_2(\mathbb F_p)$. Observe also that, although an easy 
upper bound on the dimension of $X_{\rm par}(K)$ can be given in terms of the 
number of generators of $\pi_1(K)$ or, equivalently, in terms of the number of 
rational tangles of $K$, at this point we are unable to establish the precise 
dimension of $X_{\rm par}(K)$. This requires some extra considerations and will 
be achieved in Theorem~\ref{theorem:dim par}. 

Finally note that the characters that we have constructed belong to components
of $X_{\rm par}(\Gamma)=X_{\rm par}(K)\cap X(\Gamma)$, which a priori can be 
contained in some larger components of $X_{\rm par}(K)$. We will see later that 
the components of $X_{\rm par}(K)\cap X(\Gamma)$ are in fact components of 
$X_{\rm par}(K)$ (see Lemma~\ref{lemma:K_par_in_Gamma_par}).

Remark~\ref{remark:abelian_par} allows to construct parabolic representations 
of $\pi_1(K')$ that are irreducible on some of the 
$\pi_1(B(\frac{\beta_i}{\alpha_i} ))$ and abelian on the others.

\begin{Remark}
\label{remark:small par}
Reasoning as in Remark~\ref{remark:small components}, one can see that it is
possible to construct other parabolic components of $X_{\rm par}(\Gamma)$ of 
smaller dimension by choosing some of the $\rho_i$s to be abelian: if $\ell$ 
representations among the $n$ are abelian, with $0\le \ell\le n-2$, the 
resulting components of $X_{\rm par}(\Gamma)$ have dimension $\ge n-\ell-2$. 
Note that if $\ell=n$ we obtain a point corresponding to the abelian parabolic 
character. On the other hand, the case $\ell=n-1$ is impossible, because the 
images of the two meridians in an irreducible representation cannot commute.   
\end{Remark}

\section{The non-parabolic case}
\label{section:non parabolic}

We turn now to consider the case of non-parabolic representations of a
Montesinos knot of Kinoshita-Terasaka type $K=M(\frac{\beta_1}{\alpha_1},\dots,
\frac{\beta_n}{\alpha_n},\frac{\beta_{n+1}}{\alpha_{n+1}})$ arising from
representations of the $2$-bridge knots $B(\frac{\beta_i}{\alpha_i})$, 
$i=1,\dots,n$. The construction will be similar to the one seen in the 
previous section. We start by summarising some properties of representations 
for $2$-bridge knots.

\begin{Proposition}
\label{remark:2-bridge}
The character variety of a $2$-bridge knot $X(B(\frac{\beta}{\alpha} ))$ is a 
union of plane curves ${\mathcal C}_0\cup\cdots \cup {\mathcal C}_ r\subset 
\mathbb C^2$, $r\geq 1$. Moreover, the map
$$
\tau_{\mu} : \mathcal C_j \to \mathbb C,
$$
where $\mu$ denotes a meridian, is proper.

The reducible characters form a component 
$\mathcal C_ 0= X_{\rm red}(B(\frac{\beta}{\alpha} ))$ such that 
$\tau_{\mu} :\mathcal C_0\to \mathbb C$ is an isomorphism. 
\end{Proposition}

\begin{proof}
The components are plane curves by  Corollary~\ref{cor:plane curve}.
Properness of $\tau_\mu$ means that whenever a sequence 
$\chi_n\in X(B(\frac{\beta}{\alpha} ))$ goes to infinity, then 
$\tau_{\mu}(\chi_n)=\chi_n({\mu})\to \infty$: this is a consequence of the fact 
that $\mu$ is not a boundary slope (cf. \cite{HatcherThurston}). For the second 
assertion, just notice that each reducible character is also the character of 
an abelian representation, and the abelianisation of a knot group is 
$\mathbb Z$, generated by the representative of $\mu$.
\end{proof}

It follows from Lemma~\ref{lemma:rootsAlex} that for $j>0$ the component 
$\mathcal C_j$ contains only a finite number of reducible characters, because 
the Alexander polynomial has a finite number of zeros. Thus  for almost every 
value of $\tau\in{\mathbb C}\setminus \{\pm2\}$ and for all $1\le i\le n$ there 
is an irreducible representation $\rho_i$ of 
$\pi_1(B(\frac{\beta_i}{\alpha_i}))$ such that $\chi_{\rho_i}(\mu_i)=\tau$. 
Since any two matrices of $SL_2({\mathbb C})$ having the same trace 
$\tau\neq\pm 2$ are conjugate, it follows easily that the $\rho_i$s can be 
matched together to give a representation of the composite knot $K'$. It 
follows at once that for each choice of irreducible $1$-dimensional components 
$Z_1\subset X(B(\frac{\beta_1}{\alpha_1}))$,..., 
$Z_n\subset X(B(\frac{\beta_n}{\alpha_n}))$, each containing irreducible 
characters, one can construct an irreducible component ${\mathcal C}$ of 
$X(K')$. The bending argument of Corollary~\ref{corollary:bending} shows that 
the dimension of ${\mathcal C}$ is at least $n$, the extra dimension with 
respect to the parabolic case coming from the fact that $\tau$ is a free 
parameter.

We now want to show that ${\mathcal C}\cap X(\Gamma)$ is non-empty. The 
argument will follow the same lines of Lemma~\ref{lemma:rep par Gamma}. 
However, since the elements $\rho(\mu_1)$ and $\rho(\mu_{n+1})$ are not 
parabolic, they commute if and only if they have the same axis (cf.
Section~\ref{section:bending}). Thus this time we need to keep track of two 
points in $\hat{\mathbb C}$. As in the parabolic case, one of these two points 
can be moved in an arbitrary way, however the position of the second point will 
be determined by the position of the first, because cross-ratios are preserved 
by the $SL_2({\mathbb C})$-action on $\hat{\mathbb C}$. The following 
elementary observations will be useful.

\begin{Lemma}
\label{lemma:cross-ratio}

\begin{itemize}

\item For each $\lambda \in {\mathbb C}\setminus \{0,1\}$, the subgroup of 
$SL_2({\mathbb C})$ which fixes pointwise $a$ and $b$ in $\hat{\mathbb C}$ acts simply
transitively on the pairs of distinct points $p$, $q$ of 
$\hat{\mathbb C}\setminus \{a,b\}$ such that the cross-ratio
$[a,b,p,q]=\lambda$.

\item Let $a, b, c, d\in\hat{\mathbb C}$ be four pairwise different points. For all $\lambda_1,\lambda_2\in {\mathbb C}\setminus \{0,1\}$, there are two points
$x\neq y\in\hat{\mathbb C}\setminus \{a,b,c,d\}$ such that
$[a,b,x,y]=\lambda_1$ and $[c,d,x,y]=\lambda_2$.
\end{itemize}
\end{Lemma}

\begin{proof}
The first part of the lemma follows from the fact that the subgroup of
$SL_2({\mathbb C})$ which fixes pointwise two points of $\hat{\mathbb C}$ acts 
simply transitively on the remaining points and the fact that once $a$, $b$ and 
$p$ are fixed there is a unique $q$ such that $[a,b,p,q]=\lambda$.

For the second part, without loss of generality, we may assume that $a=0$,
$b=\infty$ and $c=1$. We must find two points $x$ and $y$ such that
$\lambda_1=[0,\infty,x,y]=\frac{x}{y}$ and
$\lambda_2=[1,d,x,y]=\frac{(1-x)(d-y)}{(d-x)(1-y)}$. From the first condition
we get $x=y\lambda_1$. Replacing in the second we get a polynomial equation of
degree $2$ in the unknown $y$:
\begin{equation}
\label{eqn:yfromcrossratio} 
\lambda_1(\lambda_2-1)y^2+y(1+d\lambda_1-d\lambda_2-\lambda_1\lambda_2)+
d(\lambda_2-1)=0.
\end{equation}
This equation always admits a solution in ${\mathbb C}$ since
$\lambda_1(\lambda_2-1)\neq 0$ by hypothesis. We still need to verify that the
solution is admissible. Note that $y=0$ cannot be a solution for 
$d(\lambda_2-1)\neq 0$, $y=1$ cannot be a solution for $d(\lambda_1-1)\neq 1$, 
and $y=d$ cannot be a solution for $d(1-d)(1-\lambda_1)\lambda_2\neq 0$. 
Similarly one sees that $x$ cannot be equal to $0$, $1$ or $d$.
\end{proof}

The fact that Equation~\ref{eqn:yfromcrossratio} has two solutions may be
understood in terms of symmetries as follows. Consider the  rotation $r$ of
angle $\pi$ in hyperbolic space $\mathbf H^3$, that permutes $a$ and $b$, as 
well as $c$ and $d$. Thus if $\overline{ab}$ and $\overline{cd}\subset \mathbf 
H^3$ denote the hyperbolic geodesics with respective end-points $a$ and $b$, 
and $c$ and $d$, then $r$ is the $\pi$-rotation around the geodesic
perpendicular to $\overline{ab}$ and $\overline{cd}$. Moreover $r$ induces an
involution on each of these geodesics that reverses the orientation. 

It is easy to check that if $(x,y)$ satisfies the second assertion of
Lemma~\ref{lemma:cross-ratio}, then so does $(r(y),r(x))$, and those are all
solutions. Namely, in the the notation of the proof ($a=0$, $b=\infty$, $c=1$ 
and $d=d$) $r$ is the M\"obius transformation of 
$\hat{\mathbb{C}}=\mathbb{C}\cup\{\infty\}$, $z\mapsto d/z$ for all 
$z\in \hat{\mathbb{C}}$. Therefore $r(y)=d/y$ and $r(x)=d/x=\frac{d}{\lambda_1 
y}$. Notice that the product of the solutions of 
Equation~\ref{eqn:yfromcrossratio} is precisely $y \, r(x)= d/ \lambda_1$.

\begin{Remark}
\label{remark:2solutions_yfromcrossratio} 
There are precisely two ordered pairs of points satisfying the second assertion
of Lemma~\ref{lemma:cross-ratio}, $(x, y)$ and $(r(y), r(x))$, where $r$ is the 
hyperbolic rotation of order two that satisfies $r(a)=b$ and $r(c)=d$.
\end{Remark}

With the same notation as in the previous section we have:

\begin{Lemma}
\label{lemma:rep non-par Gamma}
For each $\rho$ there is an $a\in A_1\times\dots\times A_{n-1}$, such that 
$a\rho(\mu_1)$ and $a\rho(\mu_{n+1})$ commute.
\end{Lemma}

\begin{proof}
Note that as in the parabolic case the two elements $\rho_i(\mu_i)$ and
$\rho_i(\mu_{i+1})$ acting on $\hat{\mathbb C}$ have no fixed point in common,
for the representation $\rho_i$ is irreducible by hypothesis. Assume that
$n=3$. Let $a$ and $b\in\hat{\mathbb C}$ be the fixed points of 
$\rho_2(\mu_2)$, and $c$ and $d\in\hat{\mathbb C}$ be the fixed points of 
$\rho_2(\mu_3)$. Let $p,q \in\hat{\mathbb C}$ and $r,s \in\hat{\mathbb C}$ be 
the fixed points of $\rho_1(\mu_1)$ and $\rho_3(\mu_4)$ respectively. We define
$\lambda_1=[a,b,p,q]$ and $\lambda_2=[c,d,r,s]$. The previous lemma tells that
there is an element in the centraliser of $\rho_2(\mu_2)$ and one in the
centraliser of $\rho_2(\mu_3)$ that conjugate $\rho_1(\mu_1)$ and
$\rho_3(\mu_4)$ respectively to elements with the same fixed points. As a
consequence, one can find a representation $a\rho$ such that $a\rho(\mu_1)$ and
$a\rho(\mu_{n+1})$ commute.

If $n>3$, consider the elements $\rho_1(\mu_2)$ and $\rho_n(\mu_n)$: if they
have no common fixed point it suffices to apply verbatim the argument seen for
$n=3$. Otherwise, one can start by conjugating $\rho_{n-1}$ and $\rho_n$ by an
element in the centraliser of $\rho_{n-1}(\mu_{n-1})$ to make sure that this is
indeed the case.
\end{proof}

\begin{Remark}
\label{remark:equalorinverse}
Two hyperpolic isometries which are conjugate and have the same axis are either 
equal or inverses of one another. By appropriately choosing the order of the 
end-points of the axis of $\rho_1(\mu_1)$ and $\rho_n(\mu_{n+1})$ we can ensure 
that $a\rho(\mu_1)$ and $a\rho(\mu_{n+1})$ satisfy either one of the situations 
above. We will always assume we have made the choice that 
$a\rho(\mu_1)= a\rho(\mu_{n+1})$.
\end{Remark}

\begin{Proposition}
\label{proposition:exoticcomponents}
Let $K$ be a Montesinos knot of Kinoshita-Terasaka type with $n+1$ tangles,
$n\ge3$. The subvariety $X(\Gamma)$ of $X(K)$ contains components of dimension 
$\geq n-2$ on which the trace of the meridian is non-constant.

In addition, the number of such components can be arbitrarily large.
\end{Proposition}

\begin{proof}
We know that $X(K')$ contains irreducible components ${\mathcal C}$ of 
dimension  $n$ on which the trace of the meridian is non-constant. 
Lemma~\ref{lemma:rep non-par Gamma} ensures that the intersection 
${\mathcal C}\cap X(\Gamma)$ is non-empty. Since we need to impose two 
conditions to the points of the components of ${\mathcal C}$ for them to belong
to $X(\Gamma)$ we see that we obtain in $X(\Gamma)\subset X(K)$ components of
dimension at least $n-2$. Note that the construction shows that even on this
components the trace of the meridian is not constant.

For the last assertion, we use again that the number of irreducible
$1$-dimensional components of a $2$-bridge knot containing irreducible
characters can be arbitrarily large, according to \cite{ORK}.
\end{proof}

For later use, we discuss the space of solutions in Lemma~\ref{lemma:rep
non-par Gamma}. We start with the case $n=3$. Assume that 
$\rho=\rho_1*\rho_2*\rho_3\in R(\Gamma)$ is a representation so that $\rho_1$, 
$\rho_2$ and $\rho_3$ are irreducible. In particular $a=(a_1,a_2)=(Id,Id)$ is a 
solution. To find further solutions, let $r_i$ denote the $\pi$-rotation of 
$\mathbf H^3$ around the geodesic perpendicular to the axes of both  
$\rho_i(\mu_i)$ and $\rho_i(\mu_{i+1})$ (here $\rho_3(\mu_4)=\rho_1(\mu_1)$). 
Then 
$$
^{r_1}\rho_1*^{r_2}\rho_2*^{r_3}\rho_3
$$
is also a representation of $R(\Gamma)$, since $r_i$  conjugates
$\rho_i(\mu_i)$ and $\rho_i(\mu_{i+1})$ to their inverses.
As $r_i^2= Id$, 
$
^{r_1}\rho_1*^{r_2}\rho_2*^{r_3}\rho_3
$
is conjugate to 
$$
^{r_2r_1}\rho_1*\rho_2*^{r_2r_3}\rho_3.
$$
This representation corresponds to the second solution of 
Equation~\ref{eqn:yfromcrossratio} and is obtained as explained in 
Remark~\ref{remark:2solutions_yfromcrossratio}. Indeed both $r_1$ and $r_3$ 
permute the end-points of the axis of $\rho_1(\mu_1)=\rho_3(\mu_4)$, while 
$r_2$ is the rotation $r$ in Remark~\ref{remark:2solutions_yfromcrossratio}.
Since this representation is also conjugate to
$$
\rho_1*^{r_1r_2}\rho_2*^{r_1r_2r_2r_3}\rho_3,
$$
we have that $(a_1,a_2)= (r_1r_2,r_2r_3)$ is another solution different from 
the trivial one. It follows moreover from the discussion in
Remark~\ref{remark:2solutions_yfromcrossratio} that these are all solutions.

For larger $n$, there are more indeterminacies, but by the same argument we get
the following lemma:

\begin{Lemma}
\label{lemma:uniqueness non-par Gamma}
Assume that $a\rho$ satisfies Lemma~\ref{lemma:rep non-par Gamma} and the 
$a_1,\ldots,a_{n-3}$ are chosen generically,  so that the group generated by  
$a\rho(\mu_1)$ and $ a\rho(\mu_{n-1})$ is irreducible.
Let $r$, $r'$ and $r''$ be hyperbolic rotations of order two, so that the 
axis of $r$ is perpendicular to the axes of $a\rho(\mu_1)$ and 
$a\rho(\mu_{n-1})$, the axis of $r'$ to the axes of $a\rho(\mu_{n-1})$ and 
$a\rho(\mu_{n})$, and similarly for $r''$ and $a\rho(\mu_{n})$ and 
$a\rho(\mu_{n+1})= a\rho(\mu_{1})$. Once the $a_1,\ldots,a_{n-3}$ are fixed, 
the only other solution for the parameters $a_{n-2}$ and $a_{n-1}$ is 
$a_{n-2}'=r \, r'a_{n-2}$ and $a_{n-1}'=r' r'' a_{n-1}$.
\end{Lemma}

\section{Bounding dimensions from above}
\label{section:bound}

We have seen that it is relatively easy to establish lower bounds on the
dimension of the non-standard components we constructed in the previous
sections. Determining their exact dimension, which turns out to coincide with
the lower bound, requires a finer analysis which will be carried out in this 
section. In Subsection~\ref{subsect:conv} we give a sufficient condition to 
guarantee convergence of characters in $X(\Gamma)$, once we know that the 
restrictions to the $2$-bridge factors converge. 
Subsection~\ref{subsect:nonparabolic} deals with the non-parabolic case, and Subsection~\ref{section:more_on_intersections} with the parabolic one.

\subsection{Convergence of characters and displacement function}
\label{subsect:conv}

The goal of this subsection is to prove Proposition~\ref{prop:convergence} 
about convergence of characters. Proposition~\ref{prop:convergence} admits an 
elementary proof when the limiting characters $\chi_i^\infty$ are non 
parabolic. This follows from Lemma~\ref{lemma:cross-ratio} and the continuity 
of the only two solutions of Equation~\ref{eqn:yfromcrossratio}. Since 
Equation~\ref{eqn:yfromcrossratio} does not apply to the parabolic case, we 
need a different argument when the limit is parabolic; in fact we are going to 
give an argument that holds in general. For this purpose, first we need to 
recall the definition of the displacement function of an isometry in hyperbolic 
space and its main properties.

\begin{Definition}
Let $h\in\operatorname{Isom}(\mathbf H^3)$ be an isometry. 
Its \emph{displacement function} is
$$
\begin{array}{rcl}
 d_h:\mathbf H^3 & \to & \mathbb R_{\geq 0} \\
  x & \mapsto & d_h(x)=d(x,\gamma(x))
\end{array}.
$$ 
\end{Definition}

\begin{Lemma}
\label{lemma:prs_disp}

 \begin{itemize}
  \item[(i)] For every isometry $h\in \operatorname{Isom}(\mathbf H^3)$, $d_h$ 
is convex.
  \item[(ii)] For every $h\in \operatorname{Isom}(\mathbf H^3)$ and 
$x,y\in\mathbf H^3$, 
		  $$
			\vert d_h(x)-d_h(y)\vert\leq 2 d(x,y).
		  $$
  \item[(iii)] Let $(h_k)_{k\in\mathbb N}\subset 
\operatorname{Isom}(\mathbf H^3)$ be a sequence of isometries. If $h_k$ 
converges then $(d_{h_k}(x))_{k\in\mathbb N}$ is bounded for every 
$x\in \mathbf H^3$.
  \item[(iv)] Let $(h_k)_{k\in\mathbb N}\subset 
\operatorname{Isom}(\mathbf H^3)$ be a sequence of isometries. If there is 
$x\in\mathbf H^3$ so that $(d_{h_k}(x))_{k\in\mathbb N}$ is bounded, then 
$(h_k)_{k\in\mathbb N}$ has a convergent subsequence.
 \end{itemize}
\end{Lemma}

\begin{proof}
Assertion (i) is consequence of convexity of the distance function in 
hyperbolic space. Assertion (ii) is a straightforward application of the 
triangle inequality, and (iii) follows from continuity. Finally, (iv) follows 
from the fact that 
$$
\begin{array}{rcccl}
 O(3) & \to & \operatorname{Isom}(\mathbf H^3) & \to & \mathbf H^3 \\
      &     &           h                      & \mapsto & h(x)
\end{array}
$$
is a fibre bundle with compact fibre.
\end{proof}

The following corollary is based on Assertion (iv):

\begin{Corollary}
\label{cor:convergence} 
Let $G=\langle g_1,\ldots, g_s\mid r_i\rangle$ be a finitely generated group 
and $(\rho^k)_{k\in\mathbb N}\in R(G)$ a sequence of representations.
If there exists $x\in\mathbf H^3$ such that 
$
(\sum_{i=1}^s d_{\rho^k(g_i)}(x))_{k\in\mathbb N}
$  
is uniformly bounded, then $(\rho^k)_{k\in\mathbb N}$ has a convergent 
subsequence.
\end{Corollary}

\begin{Proposition}
\label{prop:convergence}
For $i=1,\ldots,n$, let $(\chi^k_i)_{k\in\mathbb N}$ be a sequence in 
$X(B(\beta_i/\alpha_i))$ converging to an irreducible character 
$\chi_i^\infty$. Assume that $\chi^k_1(\mu)=\cdots=\chi^k_n(\mu)\neq\pm 2$. 
Then there exist $\chi^k\in X(\Gamma)$ such that $\chi^k$ restricted to 
$\pi_1(B(\beta_i/\alpha_i))$ equals $\chi^k_i$ and $(\chi^k)_{k\in\mathbb N}$ 
converges up to a subsequence.
\end{Proposition}

Notice that even if the $\chi^k_i$ are characters of representations $\rho^k_i$
such that the sequences $( \rho^k_i)_{k\in\mathbb N}$ converge, the conjugating 
matrices in the amalgam  between $\rho^k_i$ and $\rho^k_{i+1}$ could go to 
infinity.

\begin{proof} 
Assume first that $n=3$. Let $\rho_i^k\in R(B(\beta_i/\alpha_i))$ be a
representation with character $\chi_i^k$, for $i=1,2,3$ and $k\in\mathbb N$.
Since $\chi_i^\infty$ is irreducible, there exists $\rho_i^\infty\in
R(B(\beta_i/\alpha_i))$, unique up to conjugacy, with character 
$\chi_i^\infty$. In particular, after conjugacy, we can assume that the 
sequence $(\rho_i^k)_{k\in\mathbb N}$ converges to $\rho_i^\infty$, 
for $i=1,2,3$.

Reasoning as in Lemma~\ref{lemma:rep non-par Gamma}, we see that we can find 
isometries $h_k,g_k\in \operatorname{Isom}^+(\mathbf H^3)$ for each 
$k\in\mathbb N$ such that the following three conditions are fulfilled:
\begin{eqnarray*}
h_k^{-1} \rho^k_1(\mu_2) h_k & = & \rho^k_2(\mu_2) \\
g_k^{-1} \rho^k_3(\mu_3) g_k & = & \rho^k_2(\mu_3) \\
 h_k^{-1} \rho^k_1(\mu_1) h_k & = &  g_k^{-1} \rho^k_3(\mu_4) g_k  .
\end{eqnarray*}
As a consequence, for each $k\in\mathbb N$ we are able to construct a 
representation $\rho^k$, with character $\chi^k$. We want to exploit 
Lemma~\ref{lemma:prs_disp} to prove that there is a $(\rho'^k)_{k\in\mathbb N}$ 
which converges up to a subsequence, where for each $k$ $\rho'^k$ is conjugate 
to $\rho^k$. In particular, $(\chi^k)_{k\in\mathbb N}$ converges up to a 
subsequence.

We fix a point $x\in\mathbf H^3$. According to Lemma~\ref{lemma:prs_disp}(iii) 
we have that the sequences 
$$
\begin{array}{l}
(d_{\rho^k(\mu_1)}(h_k^{-1}(x)) = d_{\rho_1^k(\mu_1)}(x))_{k\in\mathbb N} \\
(d_{\rho^k(\mu_2)}(h_k^{-1}(x)) = d_{\rho_1^k(\mu_2)}(x))_{k\in\mathbb N} \\
(d_{\rho^k(\mu_2)}(x) = d_{\rho_2^k(\mu_2)}(x)))_{k\in\mathbb N} \\
(d_{\rho^k(\mu_3)}(x) = d_{\rho_2^k(\mu_3)}(x)))_{k\in\mathbb N} \\
(d_{\rho^k(\mu_3)}(g_k^{-1}(x)) =  d_{\rho_3^k(\mu_3)}(x))_{k\in\mathbb N} \\
(d_{\rho^k(\mu_4)}(g_k^{-1}(x)) = d_{\rho_3^k(\mu_4)}(x))_{k\in\mathbb N}
\end{array}
$$
are bounded by some constant $C>0$.

We are looking for a sequence $(y_k)_{k\in\mathbb N}\subset \mathbf H^3$ such 
that $d_{\rho^k(\mu_i)}(y_k)$ is bounded above independently of $k$, 
for $i=1,2,3$. For each $k\in\mathbb N$, we consider the hyperbolic triangle 
with vertices $x$, $h_k^{-1}(x)$ and $g_k^{-1}(x)$. Thinness of hyperbolic 
triangles says that there is a point $y_k$ whose distance to each edge of this 
triangle is less than $\log(2+\sqrt 3)$. Let $y_k$ be such point. To prove the 
upper bound for $d_{\rho^k(\mu_1)}=d_{\rho^k(\mu_4)}$ on $y_k$, we notice that 
$d_{\rho^k(\mu_1)}$ is $\leq C$ on the segment between $h_k^{-1}(x)$ and 
$g_k^{-1}(x)$, by convexity. Thus, using Lemma~\ref{lemma:prs_disp}(ii),
$$
d_{\rho^k(\mu_1)}(y_k)\leq C+ 2\log(2+\sqrt 3)= C'.
$$
By a similar argument we bound $d_{\rho^k(\mu_2)}$ (using the segment between 
$h_k^{-1}(x)$ and $x$) and $d_{\rho^k(\mu_3)}$ (using the segment between $x$ 
and  $g_k^{-1}(x)$). Once we have that  $d_{\rho^k(\mu_i)}(y_k)\leq C'$, for 
each $k\in\mathbb N$ let $\rho'^k$ be the conjugate of $\rho^k$ by an isometry 
that maps $y_k$ to a fixed point $y_0$, so that 
$d_{\rho'^k(\mu_i)}(y_0)=d_{\rho^k(\mu_i)}(y_k) \leq C'$. By 
Corollary~\ref{cor:convergence}, $(\rho'^k)_{k\in\mathbb N}$ has a convergent 
subsequence.

For $n>3$ we proceed by induction on $n$. By the induction hypothesis, there is 
a convergent sequence of representations $(\phi^k)_{k\in\mathbb N}$ of the 
subgroup generated by the meridians $\mu_1,\ldots,\mu_n$ such that the 
restriction to $\langle \mu_i,\mu_{i+1}\rangle$ is conjugate to $\rho_i^k$ and
$\phi^k(\mu_1)=\phi^k(\mu_n)$. We only need its restriction $\varphi^k$ to 
$\langle \mu_2,\ldots,\mu_n \rangle$. Notice that the irreducibility of
$\rho_1^k$ implies that $\varphi^k(\mu_2)$ and $\varphi^k(\mu_n)=\phi^k(\mu_1)$ 
generate an irreducible representation. Now we can repeat the argument for 
$n=3$ applied to $\rho^k_1$, $\varphi^k$ and $\rho^k_n$.
\end{proof}

\subsection{The non-parabolic case}
\label{subsect:nonparabolic}

To bound dimensions of the components of $X(K)$ constructed in the previous 
section, we will use their intersection with the Teichm\"uller part (the
set of the characters of the basis of the Seifert fibred orbifold). Namely, in 
Lemma~\ref{lemma:interscettr0} we will prove that the components intersect the 
hyperplane $\tau_{\mu}=0$, and in Lemma~\ref{lemma:upper_bound} we will bound 
the dimension of the intersection with the Teichm\"uller part.

\begin{Lemma}
\label{lemma:interscettr0}
Let $Z\subset X(\Gamma)$ be an irreducible component as in 
Proposition~\ref{proposition:exoticcomponents}. If $\mu$ denotes a meridian, 
then $Z\cap \{\tau_{\mu}=0\}\neq\emptyset$.
\end{Lemma}

\begin{proof}
By construction, the restriction $Z\to X(B(\frac{\beta_i}{\alpha_i}))$ is 
non-constant, for each $i=1,\ldots,n$. In particular the Zariski closure of the 
image of $Z$ in $X(B(\frac{\beta_i}{\alpha_i}))$ is a curve 
$Z_i\subset X(B(\frac{\beta_i}{\alpha_i}))$. Consider the algebraic set
$$
\{(\chi_1,\ldots,\chi_n)\in Z_1\times\cdots\times Z_n
\mid \chi_1(\mu_1)=\cdots =
\chi_n(\mu_n)\}
$$
and the projection induced by taking restrictions:
$$
\pi: Z\to    
 \{(\chi_1,\ldots,\chi_n)\in Z_1\times\cdots\times Z_n
\mid \chi_1(\mu_1)=\cdots =
\chi_n(\mu_n)\}.
$$
By Proposition~\ref{remark:2-bridge}, the closure of the image 
$\overline{\pi(Z)}$ is a curve and contains a point 
$(\chi_1,\ldots,\chi_n)\in \overline{\pi(Z)}$ with $\tau_{\mu}(\chi_i)=0$. 
Since $-1$ is not a root of the Alexander polynomial, $\chi_i$ is irreducible
(a matrix with determinant $1$ and trace $0$ has eigenvalues $\pm i$, hence 
irreducibility comes from Lemma~\ref{lemma:rootsAlex}).

We take $n$ sequences of characters  
$(\chi_i^k)_{k\in\mathbb N}\subset X(B(\frac{\beta_i}{\alpha_i}))$ so that 
$\chi_i^k\to\chi_i$ as $k\to\infty$ and 
$(\chi_{\rho_1}^k,\ldots,\chi_{\rho_n}^k)\in{\pi(Z)}$. Since $\chi_i$ is 
irreducible, we may apply Proposition~\ref{prop:convergence} to conclude that
there is a convergent sequence of characters $(\chi^k)_{k\in\mathbb N}\subset 
X(\Gamma)$, $\chi^k\to\chi$, so that 
$\pi(\chi^ k)=(\chi_1^k,\ldots,\chi_n ^k)$. In particular 
$\pi(\chi)=(\chi_1,\ldots,\chi_n)$. We want to show that $\chi^k$ may be chosen 
to belong to the irreducible component $Z$. Notice that the $\chi^k$ are 
defined using the construction of Lemma~\ref{lemma:rep non-par Gamma}, 
$\chi^k=\chi_{\rho^k}$ where $\rho^k$ is the amalgam of 
${\rho_1}^k,\ldots,{\rho_n}^k$ with conjugating matrices $a_1,\ldots,a_{n-1}$. 
The $a_1,\ldots,a_{n-3}$ can be perturbed in an open (hence Zariski dense) 
subset of $A_1,\ldots,A_{n-3}$, and once those are chosen then $a_{n-2}$ and 
$a_{n-1}$ are subject to a compatibility condition. We have seen in 
Lemma~\ref{lemma:uniqueness non-par Gamma} that there are two solutions for 
$(a_{n-2},a_{n-1})$ related by rotations around axes perpendicular to the axes 
of the meridians. Notice that this construction gives at most two irreducible 
components for the fibres of $\pi^{-1}(\chi_1^k,\ldots,\chi_n^k)$.
We can assume that the sequence $(\chi^k)$ is contained in one of these two 
components. If this component is $Z$ we are done, else using the construction 
we just recalled we can find a new sequence contained in $Z$ and which 
converges by continuity (Lemma~\ref{lemma:uniqueness non-par Gamma}). Therefore 
$\chi\in Z\cap\{\tau_{\mu}=0\}$.
\end{proof}

In the previous proof we used that the Zariski closure of the image of 
$Z\subset X(\Gamma)$ in $X(B(\frac{\beta_i}{\alpha_i}))$ is a curve $Z_i\subset 
X(B(\frac{\beta_i}{\alpha_i}))$. Let $\ell$ be the number of $Z_i$s that 
consist of abelian representations. Reasoning as in Remark~\ref{remark:small 
par}, we can prove that $\ell\neq n-1$. In addition it seems unlikely that 
$\ell=n-2$ occurs for any choice of $2$-bridge knots. This case will indeed 
occur sometimes, for example for the sum of two copies of the same knot, or 
combining this with surjections between $2$-bridge knot groups (see 
Section~\ref{section:Mattman}). When $\ell=n$ then all representations of $Z$ 
are abelian.

\begin{Lemma}
\label{lemma:upper_bound}
Let $Z$ be a component of $X(\Gamma)$ contained in a component $V$ of $X(K)$ as 
in Proposition~\ref{proposition:exoticcomponents}. Let $\ell$ be the number of 
$Z_i$ that consist of abelian characters, as above. If $\ell\leq n-3$, then  
$\dim V\leq n-2-\ell$. If $\ell=n$ or $\ell=n-2$, then $\dim V \leq 1$. 
Moreover $V=Z$.
\end{Lemma}

\begin{proof}
Embed $X(K)$ in $\mathbb C^N$ with coordinates some trace functions, according 
to Proposition~\ref{proposition:X(G)}. One of these coordinates is chosen to be 
the trace of the meridian $\mu$. By Lemma~\ref{lemma:interscettr0}, $Z$ and $V$ 
intersect the hyperplane defined by trace of the meridian equal to zero, 
therefore, since $\{\tau_\mu=0\}$ has codimension $1$ in the ambient space,
$$
\dim (V\cap \{\tau_\mu=0\})\geq \dim V-1.
$$
Any matrix in $SL_2(\mathbb C)$ with zero trace has order two in 
$PSL_2(\mathbb C)$, i.e.\ it is a rotation of angle $\pi$ in hyperbolic space,
hence every representation contained in $V\cap\{\tau_{\mu}=0\}$
factors to a representation of  $\pi_1(\mathcal O_2$ into $PSL_2(\mathbb C)$,
where $\mathcal O_2$ is the three-dimensional orbifold with underlying space 
$S^3$, branching locus $K$ and ramification index $2$. Since $K$ is a 
Montesinos knot, $\mathcal O_2$ is Seifert fibred, with basis a Coxeter 
$2$-orbifold $P^2$ on a polygon with $n+1$ vertices (one for each rational 
tangle). The representations of $Z\cap \{\tau_{\mu}=0\}$ are irreducible by 
Lemma~\ref{lemma:rootsAlex}. Hence the representations corresponding to points
of a Zariski open nonempty subset of $V\cap \{\tau_{\mu}=0\}$ (containing 
$Z\cap \{\tau_{\mu}=0\}$) are also irreducible and thus they map the fibre to 
the identity. It follows that the each component of $V\cap \{\tau_{\mu}=0\}$ 
that meets $Z$ is finite-to-one to a subvariety $W$ of 
$X(P^2, PSL_2(\mathbb C))$, and 
$$
\dim V -1\leq \dim (V\cap \{\tau_{\mu}=0\})= \dim W.
$$

Assume first that $\ell=0$. We claim that $\dim W\leq n-3$ for the components 
of $X(P^2, PSL_2(\mathbb C))$ that contain characters induced by characters of 
$Z\cap \{\tau_\mu=0\}$. The corners of $P^2$ correspond to the tangles of $K$, 
and the stabiliser of each of these corners is a dihedral group. In particular, 
the $n+1$st corner is a dihedral group of order $2\alpha_{n+1}$. The 
stabilisers of the adjacent edges are order two groups, generated by 
reflections of the plane, that in $PSL_2(\mathbb C)$ are mapped to rotations. 
Thus, the meridians of the arcs adjacent to the $n+1$st tangle are mapped to 
rotations whose axes form an angle which is an integer multiple of 
$\pi/\alpha_{n+1}$. In particular this angle is constant on the irreducible 
component $W$ of $X(P^2, PSL_2(\mathbb C))$.

\begin{figure}[h]
\begin{center}
 {
 \psfrag{p1}{$\pi/\alpha_1$}
 \psfrag{p2}{$\pi/\alpha_2$}
 \psfrag{p3}{$\pi/\alpha_3$}
 \psfrag{p4}{$\pi/\alpha_4$}
 \psfrag{p5}{$\pi/\alpha_5$}
  \includegraphics[height=3cm]{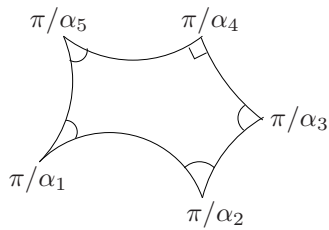}
 }
\end{center}
\caption{A Coxeter orbifold $P^2$ for $n+1=5$}
\label{fig:polygon}
\end{figure}

Now chose $W$ a component of $X(P^2, PSL_2(\mathbb C))$ that contains 
characters induced by characters in $Z\cap \{\tau_\mu=0\}$. For any 
representation $\rho\in R(P^2,PSL_2(\mathbb C) )$ coming from 
$Z\subset X(\Gamma)$, since the axes in $\mathbf H^3$ of $\rho(\mu_1)$ and 
$\rho(\mu_{n+1})$ are assumed to be the same for $\chi_\rho\in X(\Gamma)$, the 
axes of the rotations that stabilise the edges adjacent to the $n+1$st vertex 
of $P^2$ coincide and so the generators of the stabilisers are mapped to the 
same element. Therefore the dihedral stabiliser of the $n+1$st vertex is mapped 
to a group of order two. This holds true for the whole component $W$, because 
this dihedral group is finite. Thus the characters of $W$ factor through 
characters of $P'$, the Coxeter orbifold obtained by forgetting the last vertex 
of $P^2$, and $\dim W\leq \dim X(P',PSL_2(\mathbb C))$. Since $P'$ has $n$ 
vertices, $\dim X(P',PSL_2(\mathbb C))=n-3$, by 
Proposition~\ref{proposition:dimsurface}. As $\dim W\leq n-3$, 
$\dim Z\leq \dim V\leq \dim W+ 1=n-2$, and we are done when $\ell=0$.

Assume next that $\ell \leq n-3$. Then one can apply the same argument to the 
$\ell$ vertices corresponding to the $2$-bridge factors whose representations 
are abelian, and therefore we can still remove $\ell$ vertices to $P'$, to get 
the desired estimate of the dimension. The critical case occurs when 
$\ell=n-2$, the resulting $P'$ would just be a segment, and in this case the 
dimension is still zero. Finally, in the abelian case $\ell=n$, the component 
is a curve, for the abelianisation of a knot is cyclic.
\end{proof}

Using Lemma~\ref{lemma:upper_bound} and 
Proposition~\ref{proposition:exoticcomponents}, we can prove the following 
theorem. Notice that in Proposition~\ref{proposition:exoticcomponents} we found 
a lower bound for the dimension assuming that the restriction to every 
$2$-bridge factor contained irreducible representations, but a similar argument 
applies to bound the dimension when there are some abelian ones (see 
Remark~\ref{remark:small components} and 
Corollary~\ref{corollary:bending_generic}). In addition, we use \cite{ORK} to 
find arbitrarily many components.

\begin{Theorem}
\label{theorem:non parabolic}
Let $K$ be a Montesinos knot of Kinoshita-Terasaka type with $n+1$ tangles,
$n\ge3$. Then $X(K)$ contains components of dimension $d$, for $d=1,\ldots,n-2$ 
on which the trace of the meridian is non-constant, and which are entirely 
contained in $X(\Gamma)$. Moreover the number of such components is arbitrarily 
large.
\end{Theorem}

\subsection{More on intersections and the parabolic case}
\label{section:more_on_intersections}

This proposition describes how the different components we have constructed 
meet each other.

\begin{Proposition}
\label{proposition:intersection}
For all $i=1,\dots,n$, let $\rho_i$ be a parabolic representation
of the $2$-bridge knot $B(\frac{\beta_i}{\alpha_i})$ with character $\chi_i$.
For each $i$, let $Z_i$ be an irreducible component of
$X(B(\frac{\beta_i}{\alpha_i}))$ containing $\chi_i$. Denote by $Z$ the
irreducible component of $X(K)$ contained in $X(\Gamma)$ and constructed from
the $Z_i$s as in Proposition~\ref{proposition:exoticcomponents}. Let $Y$ the
parabolic component constructed from the $\rho_i$s as in
Theorem~\ref{theorem:parabolic}. One has $Z\cap Y\neq\emptyset$.
\end{Proposition}

\begin{proof}
The proof of this lemma is similar to the proof of 
Lemma~\ref{lemma:interscettr0}, because $1$ is not a root of the Alexander 
polynomial of any knot. However here we have to use that $Y$ contains 
characters of representations $\rho$ satisfying $\rho(\mu_1)=\rho(\mu_{n+1})$
and not only that $\rho(\mu_1)$ and $\rho(\mu_{n+1})$ commute. 
We suppose first that $n=3$. Let $\rho$ be a representation with character 
$\chi_\rho\in Y$. Let $\rho_1$, $\rho_2$ and $\rho_3$ be the restrictions of 
$\rho$. We have that, up to conjugacy, the parabolic transformation 
$\rho_1(\mu_2)= \rho_2(\mu_2)$ fixes $\infty\in\hat{\mathbb C}$, 
$\rho_2(\mu_3)=\rho_3(\mu_3)$ fixes $0$, and $\rho_1(\mu_1)$ and 
$\rho_3(\mu_4)$ fix $1$. This can be achieved thanks to
Lemma~\ref{lemma:rep par Gamma}. We want to conjugate further $\rho_1$ and 
$\rho_3$ by matrices $g,h\in PSL_2(\mathbb C)$ so that 
$g\rho_1(\mu_1)g^{-1}=h\rho_3(\mu_4)h^{-1}$ (i.e. we want them to be equal, not 
only commuting), and so that $g\rho_1(\mu_2)g^{-1}=\rho_1(\mu_2)$ and 
$h\rho_3(\mu_3)h^{-1}=\rho_3(\mu_3)$. Thus we have to choose
$$
g=\pm \begin{pmatrix}
       1 & x-1 \\
      0  & 1
      \end{pmatrix}
\ \textrm{ and }\ 
h=\pm \begin{pmatrix}
       1 & 0 \\
      y-1  & 1
      \end{pmatrix},
$$
for $x,y\in\mathbb C$. Since $\rho_1(\mu_1)$ and $\rho_3(\mu_4)$ are parabolic 
matrices that fix $1$,
$$
\rho_1(\mu_1)= \begin{pmatrix}
       1 +a & -a \\
      a  & 1 -a
      \end{pmatrix}
\ \textrm{ and }\ 
\rho_3(\mu_4)= \begin{pmatrix}
       1 +b & -b \\
      b  & 1 -b
      \end{pmatrix}
$$
where $a,b\in \mathbb C\setminus \{0\}$.
A straightforward computation shows that the equation
$$
g \rho_1(\mu_1) g^{-1}=h \rho_3(\mu_4) h^{-1}
$$
has solutions:
$$
x=\pm \sqrt{b/a}\textrm{ and } y=1/x=\pm \sqrt{a/b}.
$$
The resulting matrices are:
$$
g\rho_1(\mu_1)g^{-1}=h\rho_3(\mu_4)h^{-1}=
\begin{pmatrix}
       1 \pm\sqrt{ab} & -b \\
      a  &  1 \mp\sqrt{ba}
      \end{pmatrix}.
$$
Notice that changing the sign of the square root corresponds to conjugating:
$$
\begin{pmatrix}
       1 -\sqrt{ab} & -b \\
      a  &  1 +\sqrt{ba}
      \end{pmatrix}
= 
R
\begin{pmatrix}
       1 +\sqrt{ab} & -b \\
      a  &  1 -\sqrt{ba}
      \end{pmatrix}^{-1}
R^{-1}$$
where 
$$
R= \begin{pmatrix}
       -i & 0 \\
      0  &    i
      \end{pmatrix}
$$
is the matrix of a rotation around the axis with end-points $0$ and $\infty$,
that are precisely the points in $\hat{\mathbb C}=\mathbb C\cup\{\infty\}$
fixed by $\rho_1(\mu_2)$ and $\rho_3(\mu_3)$. This relation by a rotation can 
be seen as the limit of the rotations that appear for the non-parabolic
representations in Remark~\ref{remark:2solutions_yfromcrossratio} and
Lemma~\ref{lemma:uniqueness non-par Gamma}. This guarantees that the parabolic 
representations are the limit of non-parabolic representations in $Z$ provided 
by Proposition~\ref{prop:convergence}.

Similarly, for $n\geq 4$, we apply the openness argument to the $n-3$ 
conjugating matrices $a_1\ldots,a_{n-3}$ in $A_1,\ldots,A_{n-3}$: $a_i$ can be 
chosen in an open (Zariski dense) subset of $A_i$, for $i=1,\cdots,n-1$. Then 
we apply again the  the previous argument to $a_{n-2}$ and $a_{n-1 }$.
\end{proof}

\begin{Lemma}
\label{lemma:K_par_in_Gamma_par}
Let $Y\subset  X_{\rm par}(K)$ be an irreducible component such that $Y \cap
X(\Gamma)\neq\emptyset$. Assume there is a character $\chi_{\rho_0}\in Y \cap
X(\Gamma)$ such that for three meridians $\mu_i$, $\mu_j$ and $\mu_k$ the
points of $\hat{\mathbb C}$ fixed by $\rho_0(\mu_i)$, $\rho_0(\mu_j)$ and 
$\rho_0(\mu_k)$ are all different. Then $Y \subset X(\Gamma)$.
\end{Lemma}

\begin{proof}
Let $\rho_0\in R_{\rm par}(K)$ be a parabolic representation with character
$\chi_0\in Y \cap X(\Gamma)$ satisfying the hypothesis of the lemma.
Seeking a contradiction, we assume that $Y\cap X(\Gamma)$ is not equal to $Y$.
Then, by the curve selection lemma, there exists a deformation 
$\rho_{s}\in R_{\rm par}(K)$ of $\rho_0$, analytic in 
$s\in (-\varepsilon,\varepsilon)$, such that 
$\chi_{\rho_s}\in Y\setminus (Y \cap X(\Gamma))$ for $s\neq 0$.
Using the notation of Figure~\ref{fig:composite}, we have the following
relations
$$
\mu_1^{-1}\mu_1'=\mu_2^{-1}\mu_2'=\cdots = \mu_{n+1}^{-1}\mu_{n+1}'.
$$
Let $f= \mu_l^{-1}\mu_l'$, we have that $\rho_0(f)$ is the identity. We claim
that $\rho_s(f)$ is also trivial for $s\in (-\varepsilon,\varepsilon)$.
Otherwise, if $\rho_s(f)$ was non-trivial for some $s\in
(-\varepsilon,\varepsilon)$, then, by Claim~\ref{claim:fixedpoints} below, one
of its fixed points in $\hat{\mathbb C}$ would be arbitrarily close to the 
points fixed by $\rho_0(\mu_l)= \rho_0(\mu_l')$, for each $l=1,\ldots, n+1$.
But since we assume that the points fixed by $\rho_0(\mu_i)$, $\rho_0(\mu_j)$ 
and $\rho_0(\mu_k)$ are different, and since $\rho_s(f)$ has at most two fixed 
points in $\hat{\mathbb C}$, $\rho_s(f)$ must be trivial. We deduce that 
$\rho_s(\mu_n)=\rho_s(\mu_n')$ and $\rho_s(\mu_{n+1})=\rho_s(\mu_{n+1}')$. 
In particular the restriction of $\rho_s$ factors to a representation
$$
\varphi_s:\pi_1(B(\tfrac{\alpha_{n+1}}{\beta_{n+1}}))\to SL_2(\mathbb C).
$$ 
Since, for each $s\in(-\varepsilon,\varepsilon)$, $\varphi_s$ is parabolic and
it is a deformation of a parabolic abelian representation $\varphi_0$ of 
$B(\frac{\alpha_{n+1}}{\beta_{n+1}})$, $\varphi_s$ is still abelian for each
$s\in(-\varepsilon,\varepsilon)$, by Lemma~\ref{lemma:rootsAlex}, and therefore 
$\rho_s\in R_{\rm par}(\Gamma)$ and $\chi_{\rho_s}\in X_{\rm par}(\Gamma)$. 
Hence we get a contradiction that proves the lemma, assuming
Claim~\ref{claim:fixedpoints}.
\end{proof}

\begin{Claim}
\label{claim:fixedpoints}
Let $\rho_s\in R_{\rm par}(K)$ be a deformation of $\rho_0$ as in the proof of
Lemma~\ref{lemma:K_par_in_Gamma_par}, analytic in 
$s\in (-\varepsilon,\varepsilon)$. Suppose that $\rho_s(\mu_i^{-1}\mu_i')$ is 
nontrivial for $s\neq 0$. Then at least one of the fixed points of 
$\rho_s(\mu_i^{-1}\mu_i')$ in $\hat{\mathbb C}$ converges to the fixed point of
$\rho_0(\mu_i)=\rho_0(\mu_i')$ as $s\to 0$.
\end{Claim}

\begin{proof}
We may assume that 
$$
\rho_0(\mu_i)=\rho_0(\mu_i')=\begin{pmatrix}
                              1 & 1\\ 
			      0 & 1
                             \end{pmatrix}.
$$
In addition, since $\rho_s(\mu_i)$ is parabolic, its fixed point in
$\hat{\mathbb C}$ changes analytically and, conjugating by matrices that map
it to $\infty$, we may assume that this fixed point is constant. Furthermore, 
conjugating by diagonal matrices that depend analytically on $s$, we may assume 
that $\rho_s(\mu_i)=\rho_0(\mu_i)$ remains constant in 
$s\in(-\varepsilon,\varepsilon)$. Since $\rho_s(\mu_i')$ is parabolic, we then 
write
$$
\rho_s(\mu_i')=\begin{pmatrix}
                              1+a(s) & 1+ b(s) \\
			      c(s) & 1-a(s)
                             \end{pmatrix},
$$
where $a$, $b$ and $c$ are analytic functions in $s$ satisfying 
$
a(0)=b(0)=c(0)=0$ and  $a^2 +(1+b)c=0$.
Then, a straightforward computation gives that the points of $\hat {\mathbb C}$ 
fixed by 
$$
\rho_s(\mu_i^{-1}\mu_i' )= \begin{pmatrix}
                              1+a(s)-c(s) & b(s)-a(s) \\
			      c(s) & 1-a(s)
                             \end{pmatrix}
$$ 
are: 
$$
\{z\in \hat {\mathbb C}\mid  c z^2+(c-2 a)z+ a-b= 0\}.
$$
Using $c=-a^2/(1+b)$, the sum of the two solutions of this quadratic equation 
is 
$$
\frac{2 a(s) -c(s)}{c(s)}=\frac{-2(1+b(s))}{a(s)}-1 ,
$$
that converges to infinity as $s\to 0$. Thus, for $s$ sufficiently small, at 
least one of the solutions is arbitrarily close to $\infty$, the point fixed by 
$\rho_0(\mu_i)$.
\end{proof}

We can now prove:
 
\begin{Theorem}
\label{theorem:dim par}
Let $Y$ be a component of $X_{\rm par}(K)$ constructed in 
Theorem~\ref{theorem:parabolic}. Then $\dim Y= n-2$. 
\end{Theorem}

\begin{proof}
By Lemma~\ref{lemma:K_par_in_Gamma_par}, we may assume that $Y$ is a component 
of $X_{par}(\Gamma)$. According to Proposition~\ref{proposition:intersection} 
there is a component $Z$ of $X(K)$ which intersects $Y$ and on which the trace 
of the meridian is non-constant. According to 
Proposition~\ref{proposition:intersection} there is a component $Z$ of $X(K)$ 
which intersects $Y$ and on which the trace of the meridian is non-constant. We 
have that $Z\cap Y$ is contained in the intersection of $Z$ with the 
hyperplanes defined by the condition that the trace of the meridian is equal to 
$\pm 2$. Using the fact that $\dim Z\le n-2$ (see 
Theorem~\ref{theorem:non parabolic}) and that the trace of the meridian is
non-constant on $Z$, we deduce that $\dim (Z\cap Y)\le n-3$. On the other hand,
$Z\cap Y$ is obtained from $Y$ by imposing just one condition: indeed, this is
the condition required for two parabolic matrices which commute to be the same.
It follows that $n-3\ge \dim (Z\cap Y)\ge \dim Y-1$, and $\dim Y\leq n-2$. 
The last statement follows from the fact that $\dim Y\ge n-2$.

The same dimensional bound can be obtained directly by observing that $\dim  
X_{par}(\Gamma) < \dim X_{par}(K')$, because for a generic character 
$\chi_\rho\in X_{par}(K') $, $\rho(\mu_1)$ and $\rho(\mu_{n+1})$ do not commute
(the fixed points in $\hat{\mathbb C}$ are different) and in addition, by 
Corollary~\ref{corollary:bending_par}, the dimension of the component of 
$X_{par}(K')$ containing $Y$ is $\leq n-1$.
\end{proof}

Of course, a similar result holds for the parabolic components of smaller
dimension described in Remark~\ref{remark:small par}, using the same argument,
Remark~\ref{remark:small components} and Corollary~\ref{corollary:bending_par}.

\section{Other non-standard components}
\label{section:Mattman}

In the previous sections we relied on bending to be able to construct new 
non-standard components. The commuting trick, however, allows to construct 
other non-standard components which are not obtained by bending. Consider for
instance a Montesinos knot $K$ of Kinoshita-Terasaka type with $n+1=3$ rational
tangles. In this case, the construction of Section~\ref{section:parabolic} can
be carried out, but only gives a finite number of parabolic representation up
to conjugacy. 
On the other hand, the argument of 
Section~\ref{section:non parabolic} does not apply anymore. 

It was however shown by Mattman that some of these knots admit non-standard
components on which the trace of the meridian is not constant. We start with a
simple observation.

\begin{Remark}
\label{remark:quotient}
Assume that for $i=1,2$, there is an epimorphism $\psi_i: 
\pi_1(B(\frac{\beta_i}{\alpha_i}))\longrightarrow G$ such that
$\psi_1(\mu_2)=\psi_2(\mu_2)$, and such that $\psi_1(\mu_1)$, $\psi_2(\mu_3)$
commute. Then each representation $\rho$ of $G$ into $SL_2({\mathbb C})$
induces a representation of $K$ in which the images of $\mu_1$ and $\mu_3$
commute. The induced representation is obtained by ``doubling" $\rho$.
\end{Remark}

Of course, one can adapt the reasoning in the above remark to the case of
Montesinos knots of Kinoshita-Terasaka type with more than three rational
tangles, or more generally to other knots obtained as Kinoshita-Terasaka sums.

Note that the hypothesis of Remark~\ref{remark:quotient} are trivially
satisfied when $\frac{\beta_1}{\alpha_1}=\frac{\beta_2}{\alpha_2}$ by taking
$G=\pi_1(B(\frac{\beta_1}{\alpha_1}))$. Indeed, one can show that one can 
choose a continued fraction expansion for $\frac{\beta}{\alpha}$ in which 
$a'_i=a''_1$ for all $i$.

Mattman considered the case where $\frac{\beta_1}{\alpha_1}=\frac{1}{3}$, and
$\frac{\beta_2}{\alpha_2}=\frac{1}{m}$, and found non-standard components in
the case where $m$ is a multiple of $3$. This follows from the fact that there
is a $\pi_1$-surjective map of degree $\frac{m}{3}$ from the $(2,m)$-torus knot 
onto the trefoil knot if $3$ divides $m$ (see Figure~\ref{fig:trefoil}). As a 
consequence, the character variety of these pretzel knots contains the 
character variety of the trefoil knot as non-standard component.

It is worth to point out that in general the non-standard components obtained
in this way have small dimension with respect to the non-standard components
obtained by bending.

\begin{figure}[h]
\begin{center}

  \includegraphics[height=6cm]{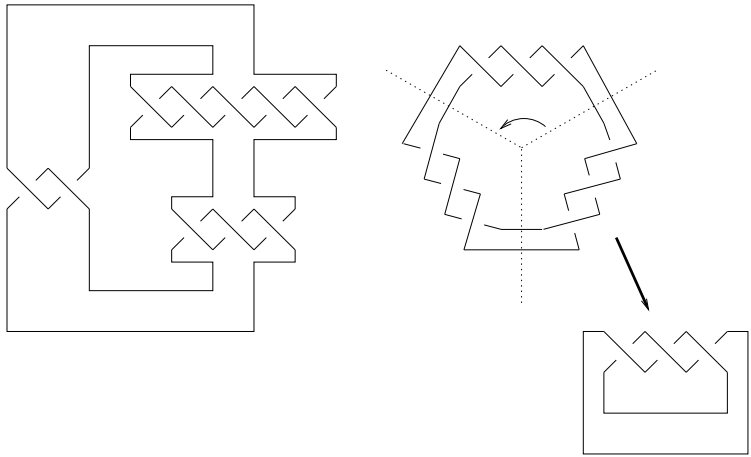}
\end{center}
\caption{One of the $(2,3,k)$ pretzel knots considered by Mattman, and a
$\pi_1$-surjective branched covering from the $(2,9)$-torus knot onto
the trefoil knot.}
\label{fig:trefoil}
\end{figure}

\section{Representations over fields of positive characteristic}
\label{section:Fp}

For an odd prime $p$, define 
$$
\Gamma_p=\Gamma/ \langle \mu^p\rangle
$$
where $\mu$ denotes a meridian as usual.

\begin{Lemma}
\label{lemma:dimgammap}
For almost all prime $p$, $\dim X(\Gamma_p) \leq n-3$. 
\end{Lemma}

\begin{proof}
By construction, 
$$
X(\Gamma_p)= X(\Gamma)\cap\{\tau_\mu= 2\cos (\tfrac{k\pi}p) 
\mid k=1,\ldots, \tfrac{p-i}2 \}.
$$ 
Hence, for almost all $p$, $X(\Gamma)\cap\{\tau_\mu= 2\cos (\frac{k\pi}p)  
\mid  k=1,\ldots, \frac{p-i}2 \}$ is contained in the union of some irreducible 
components $Z_1,\ldots, Z_r$ of $X(\Gamma)$, for which $\tau_{\mu}$ is 
non-constant. Lemma~\ref{lemma:interscettr0} and 
Theorem~\ref{theorem:non parabolic} apply to $Z_i$ and $\dim Z_i\leq n-2$.
\end{proof}

Given $p$ as in the previous lemma,   for almost every odd prime $q$,     
$\dim X(\Gamma_p)_{\overline{\mathbb F}_q} \leq n-3$. Thus the following
results tells that $X(\Gamma_p)$ ramifies at $p$:

\begin{Proposition}
\label{lemma:gammapmodp}
For almost all prime $p$, $\dim X(\Gamma_p) _{\Fp}\geq  n-2$. In particular 
$X(\Gamma_p)$ ramifies at $p$.
\end{Proposition}

\begin{proof}
Since $\Fp$ has characteristic $p$, then a representation of $\Gamma$ in 
$SL_2(\Fp)$ factors through $\Gamma_p$ iff $\mu$ is mapped to a parabolic 
element. Thus $X(\Gamma_p)_{\Fp}=X_{par}(\Gamma)_{\Fp}$. Moreover, for almost
all $p$, $X_{par}(\Gamma)_{\Fp}$ 
has the same dimension as $X_{par}(\Gamma)$, that is $\geq n-2$.
\end{proof}

Let $\mathcal O_p$ denote the orbifold with underlying space $S^3$,
singular locus $K$ and ramification of order $p$, an odd prime.
Recall that the orbifold fundamental group of   $\mathcal O_p$
is $\pi_1(S^3\setminus \mathcal N(K))/\langle \mu^p\rangle $.
The results above show that the subvariety 
 $X(\Gamma_p)_{\mathbb K}\subseteq X(\mathcal O_p)_{\mathbb K}$ 
has a larger dimension for $\mathbb K=\Fp$ than for $\mathbb K=\mathbb C$.

The extra ideal points of  $X(\Gamma_p)_{\Fp}\subseteq X(\mathcal O_p)_{\Fp}$
give rise to essential 2-suborbifolds of $\mathcal O_p$ which meet $K$.
They correspond to properly embedded essential surfaces in the 
exterior of $K$ whose boundary components are meridians. This 
is typically the case of Conway spheres.

It would be interesting to understand whether these essential
2-suborbifolds of $\mathcal O_p$ can be associated to  ideal 
points of curves in $X(\mathcal O_p)_{\mathbb K}$ for an
arbitrary ${\mathbb K}$.

\paragraph{Acknowledgements}
The authors are indebted to Ariane M\'ezard for suggesting the book of
Cox, Little and O'Shea \cite{CLoS} as a reference in arithmetic
geometry and for valuable discussion.
The second author is partially supported by the Spanish Micinn through grant MTM2009-0759 and by the Catalan AGAUR through grant SGR2009-1207. He
received the prize ``ICREA Acad\`emia''  for excellence in research, funded by the Generalitat de Catalunya. He  thanks the Universit\'e de Provence
for support and hospitality.

\begin{footnotesize}

\bibliographystyle{plain}

\textsc{Universit\'e de Provence, LATP UMR 6632 du CNRS}
	
\textsc{CMI, Technop\^ole de Ch\^ateau Gombert}

\textsc{39, rue F. Joliot Curie, 13453 Marseille CEDEX 13, France}
	
{paoluzzi@cmi.univ-mrs.fr}

\medskip

\textsc{Departament de Matem\`atiques, Universitat Aut\`onoma de Barcelona.}

\textsc{08193 Bellaterra, Spain}

{porti@mat.uab.es}

\end{footnotesize}

\end{document}